\documentclass{amsart}

\newtheorem{theorem}{Theorem}[section]
\newtheorem{lemma}{Lemma}
\newtheorem{corollary}{Corollary}
\newtheorem{proposition}{Proposition}
\theoremstyle{conjecture}
\newtheorem*{conjecture}{Conjecture}
\theoremstyle{definition}

\newtheorem*{notation}{Notation}

\theoremstyle{remark}
\newtheorem*{remark}{Remark}
\theoremstyle{remarks}
\newtheorem*{remarks}{Remarks}
\theoremstyle{example}
\newtheorem*{example}{Example}
\numberwithin{equation}{section}

\begin{document}
\title{Loops in $SU(2)$ and Factorization}

\author{Doug Pickrell}
\email{pickrell@math.arizona.edu}

\begin{abstract}We discuss a refinement of triangular factorization
for the loop group of $SU(2)$.
\end{abstract}
\maketitle

\setcounter{section}{-1}

\section{Introduction}\label{Introduction}

This paper is a sequel to \cite{P}.  The main purpose of the paper
is to prove functional analytic generalizations of Theorems
\ref{SU(2)theorem1} and \ref{U(2)theorem} below.

Let $L_{fin}SU(2)$ ($L_{fin}SL(2,\mathbb C)$) denote the group
consisting of functions $S^1 \to SU(2)$ ($SL(2,\mathbb C)$,
respectively) having finite Fourier series, with pointwise
multiplication. For example, for $\zeta \in \mathbb C$ and $n\in
\mathbb Z$, the function
$$S^1 \to SU(2):z \to
a(\zeta) \left(\begin{matrix} 1&
\zeta z^{-n}\\
-\bar{\zeta}z^n&1\end{matrix} \right),$$ where $a(\zeta)=(1+\vert
\zeta \vert ^2)^{-1/2}$, is in $L_{fin}SU(2)$. It is known that
$L_{fin}SU(2)$ is dense in $C^{\infty}(S^1,SU(2))$ (Proposition
3.5.3 of \cite{PS}). Also, if $f(z)=\sum f_n z^n$, let $f^*=\sum
\bar f_n z^{-n}$. If $f \in H^0(\Delta)$, then $f^* \in
H^0(\Delta^*)$, where $\Delta$ is the open unit disk, and
$\Delta^*$ is the open unit disk at $\infty$.

\begin{theorem}\label{SU(2)theorem1} Suppose that $k_1 \in L_{fin}SU(2)$. The
following are equivalent:

($a_1$) $k_1$ is of the form
$$k_1(z)=\left(\begin{matrix} a(z)&b(z)\\
-b^*&a^*\end{matrix} \right),\quad z\in S^1,$$ where $a$ and $b$
are polynomials in $z$, and $a(0)>0$.

($b_1$) $k_1$ has a factorization of the form
$$k_1(z)=a(\eta_n)\left(\begin{matrix} 1&-\bar{\eta}_nz^n\\
\eta_nz^{-n}&1\end{matrix} \right)..a(\eta_0)\left(\begin{matrix}
1&
-\bar{\eta}_0\\
\eta_0&1\end{matrix} \right),$$ for some $\eta_j \in \mathbb C$.

($c_1$) $k_1$ has triangular factorization of the form
$$\left(\begin{matrix} 1&0\\
\sum_{j=0}^n \bar y_jz^{-j}&1\end{matrix} \right)\left(\begin{matrix} a_1&0\\
0&a_1^{-1}\end{matrix} \right)\left(\begin{matrix} \alpha_1 (z)&\beta_1 (z)\\
\gamma_1 (z)&\delta_1 (z)\end{matrix} \right),$$ where $a_1>0$ and
the third factor is a polynomial in $z$ which is unipotent upper
triangular at $z=0$.

Similarly, the following are equivalent:

($a_2$) $k_2$ is of the form
$$k_2(z)=\left(\begin{matrix} d^{*}&-c^{*}\\
c(z)&d(z)\end{matrix} \right),\quad z\in S^1,$$ where $c$ and $d$
are polynomials in $z$, $c(0)=0$, and $d(0)>0$.

($b_2$) $k_2$ has a factorization of the form
$$k_2(z)=a(\zeta_n)\left(\begin{matrix} 1&\zeta_nz^{-n}\\
-\bar{\zeta}_nz^n&1\end{matrix}
\right)..a(\zeta_1)\left(\begin{matrix} 1&
\zeta_1z^{-1}\\
-\bar{\zeta}_1z&1\end{matrix} \right),$$ for some $\zeta_j \in
\mathbb C$.

($c_2$) $k_2$ has triangular factorization of the form
$$\left(\begin{matrix} 1&\sum_{j=1}^n \bar x_jz^{-j}\\
0&1\end{matrix} \right)\left(\begin{matrix} a_2&0\\
0&a_2^{-1}\end{matrix} \right)\left(\begin{matrix} \alpha_2 (z)&\beta_2 (z)\\
\gamma_2 (z)&\delta_2 (z)\end{matrix} \right), $$ where $a_2>0$
and the third factor is a polynomial in $z$ which is unipotent
upper triangular at $z=0$.

\end{theorem}

\begin{remark} The two sets of conditions are equivalent; they are intertwined by
the outer involution $\sigma$ of $LSL(2,\mathbb C)$ given by
\begin{equation}\label{outer}\sigma(\left(\begin{matrix}a&b\\c&d\end{matrix}\right))=
\left(\begin{matrix}d&cz^{-1}\\bz&a\end{matrix}\right).\end{equation}
\end{remark}

This Theorem basically follows from results in \cite{P}, but it is
possible to give a direct argument (not involving Lie theory). We
will present this, and functional analytic generalizations, in
Section \ref{SU(2)case}.

The terminology regarding triangular factorization in the
following theorem is reviewed in Section
\ref{triangularfactorization}.

\begin{theorem}\label{U(2)theorem}
(a) If $\{\eta_i\}$ and $\{\zeta_j\}$ are rapidly decreasing
sequences of complex numbers, then the limits
$$k_1(z)=\lim_{n\to\infty}a(\eta_n)\left(\begin{matrix} 1&-\bar{\eta}_nz^n\\
\eta_nz^{-n}&1\end{matrix} \right)..a(\eta_0)\left(\begin{matrix}
1&
-\bar{\eta}_0\\
\eta_0&1\end{matrix} \right)$$ and
$$k_2(z)=\lim_{n\to\infty}a(\zeta_n)\left(\begin{matrix} 1&\zeta_nz^{-n}\\
-\bar{\zeta}_nz^n&1\end{matrix}
\right)..a(\zeta_1)\left(\begin{matrix} 1&
\zeta_1z^{-1}\\
-\bar{\zeta}_1z&1\end{matrix} \right),$$ exist in
$C^{\infty}(S^1,SU(2))$.

(b) Suppose $g\in C^{\infty}(S^1,SU(2))$. The following are
equivalent:

(i) $g$ has a triangular factorization $g=lmau$, where $l$ and $u$
have $C^{\infty}$ boundary values.

(ii) $g$ has a factorization of the form
$$g(z)=k_1(z)^*\left(\begin{matrix} e^{\chi}&0\\
0&e^{-\chi}\end{matrix}\right)k_2(z),$$ where $\chi \in
C^{\infty}(S^1,i\mathbb R)$, and $k_1$ and $k_2$ are as in (a).

(iii) The Toeplitz operator $A(g)$ and the shifted Toeplitz
operator $A_1(g)$ are invertible.

\end{theorem}

\begin{remarks} (a) Suppose that $g \in L_{fin}SU(2)$. The $l$ and $u$
factors in (i) are also in $L_{fin}SL(2,\mathbb C)$, but they are
essentially never unitary on $S^1$. On the other hand the factors
$k_j$ in (ii) are unitary, but in general they are not in
$L_{fin}SU(2)$ [If $k_1,k_2\in L_{fin}SU(2)$, then $\chi$ must be
constant. Since $L_{fin}SU(2)$ is dense in
$C^{\infty}(S^1,SU(2))$, the parameterization in (ii) implies that
generically $g$ will correspond to nonconstant $\chi$.].

(b) There is a generalization of this Theorem with $U(2)$ in place
of $SU(2)$, where one restricts to loops in the identity
component. We will restrict our attention to $SU(2)$, to simplify
the exposition.
\end{remarks}

The outline of the paper is the following. Section
\ref{triangularfactorization} is a review of standard facts about
triangular factorization.

In Sections \ref{SU(2)case} and \ref{SU(2)caseII}, we prove
Theorems \ref{SU(2)theorem1} and \ref{U(2)theorem}, respectively.
In these two sections, the main point is to extend the
equivalences above to other function spaces, especially the
critical Sobolev space $W^{1/2,L^2}$; see Theorems
\ref{SU(2)theorem1W} and \ref{maintheoremW}.

It seems possible that there are $L^2$ generalizations of these
theorems. This is briefly discussed in Section \ref{Summary}. In
the Appendix we discuss the combinatorial relation between $x^*$
and $\zeta$ in Theorem \ref{SU(2)theorem1}. This relation is
central to the $L^2$ question, and applications. Unfortunately
this relation remains mysterious to me.

The generalization of the algebraic aspects of this paper from
$SU(2)$ to general simply connected compact groups is known
(\cite{P},\cite{PP}), but considerably more complicated. For
$SU(2)$ it suffices to consider one representation, the defining
representation, which greatly simplifies everything.

\begin{notation} Sobolev spaces will be denoted by $W^s$, and will
always be understood in the $L^2$ sense. The space of sequences
satisfying $\sum n\vert \zeta_n\vert^2<\infty$ will be denoted by
$w^{1/2}$. We will write $Meas(S^1,SU(2))$ (rather than
$L^{\infty}(S^1,SU(2))$) for the group of (equivalence classes of)
measurable maps. This group is usually equipped with the topology
of convergence in measure, but this will not play a role in this
paper.
\end{notation}

We will use \cite{Pe} as a general reference for Hankel and
Toeplitz operators.

\section{Triangular factorization for $LSL(2,\mathbb
C)$}\label{triangularfactorization}

Suppose that $g\in L^1(S^1,SL(2,{\mathbb C}))$. A triangular
factorization of $g$ is a factorization of the form
\begin{equation}\label{factorization}g=l(g)m(g)a(g)u(g),\end{equation}
where
\[l=\left(\begin{array}{cc}
l_{11}&l_{12}\\
l_{21}&l_{22}\end{array} \right)\in H^0(\Delta^{*},SL(2,{\mathbb
C})),\quad l(\infty )=\left(\begin{array}{cc}
1&0\\
l_{21}(\infty )&1\end{array} \right),\] $l$ has a $L^2$ radial
limit, $m=\left(\begin{array}{cc}
m_0&0\\
0&m_0^{-1}\end{array} \right)$, $m_0\in S^1$,
$a(g)=\left(\begin{array}{cc}
a_0&0\\
0&a_0^{-1}\end{array} \right)$, $a_0>0$,
\[u=\left(\begin{array}{cc}
u_{11}&u_{12}\\
u_{21}&u_{22}\end{array} \right)\in H^0(\Delta ,SL(2,{\mathbb
C})),\quad u(0)=\left(\begin{array}{cc}
1&u_{12}(0)\\
0&1\end{array} \right),\] and $u$ has a $L^2$ radial limit. Note
that (\ref{factorization}) is an equality of measurable functions
on $S^1$. A Birkhoff (or Wiener-Hopf, or Riemann-Hilbert)
factorization is a factorization of the form $g=g_-g_0g_+$, where
$g_-\in H^0(\Delta^*,\infty;SL(2,\mathbb C),1)$, $g_0\in
SL(2,\mathbb C)$, $g_+\in H^0(\Delta,0;SL(2,\mathbb C),1)$, and
$g_{\pm}$ have $L^2$ radial limits on $S^1$. Clearly $g$ has a
triangular factorization if and only if $g$ has a Birkhoff
factorization and $g_0$ has a triangular factorization, in the
usual sense of matrices.

\begin{proposition}Birkhoff and triangular factorizations are
unique. \end{proposition}

\begin{proof} If $g_-g_0g_+=h_-h_0h_+$ are two Birkhoff factorizations,
then the function $F$ equal to
$h_-^{-1}g_-$ for $\vert z\vert \ge 1$ and $(h_0h_+)^{-1}g_0g_+$
for $\vert z\vert \le 1$ is holomorphic on $\mathbb C\setminus
S^1$ and integrable on $S^1$. Integrability implies that the
singularities along $S^1$ are removable. Therefore $F$ is
constant, and the normalization conditions force $F=1$. This
implies uniqueness.
\end{proof}

\begin{remark} In the definition of Birkhoff factorization, if
the $L^2$ condition is replaced by the weaker condition that
$g_{\pm}$ have pointwise radial limits $a.e.$ on $S^1$, then
factorization is not unique. For example
$$\left(\begin{matrix}1&0\\0&1\end{matrix}\right)
=\left(\begin{matrix}\frac{z+1}{z-1}&0\\0&\frac{z-1}{z+1}\end{matrix}\right)
\left(\begin{matrix}-1&0\\0&-1\end{matrix}\right)
\left(\begin{matrix}-\frac{z-1}{z+1}&0\\0&-\frac{z+1}{z-1}\end{matrix}\right)$$
is a factorization in this weaker sense. At least for the purposes
of this paper, $L^2$ appears to be the natural regularity
condition in the definitions of factorization.
\end{remark}

As in \cite{PS}, consider the polarized Hilbert space
\[{\mathcal H}:=L^2(S^1,C^2)={\mathcal H}^{+}\oplus {\mathcal H}^{-},\]
where ${\mathcal H}^{+}=P_{+}{\mathcal H}$ consists of
$L^2$-boundary values of functions holomorphic in $\Delta$. If
$g\in L^{\infty}(S^1,SL(2,\mathbb C))$, we write the bounded
multiplication operator defined by $g$ on $\mathcal H$ as
\[M_g=\left(\begin{array}{cc}
A(g)&B(g)\\
C(g)&D(g)\end{array} \right)\] where $A(g)=P_{+}M_gP_{+}$ is the
(block) Toeplitz operator associated to $g$ and so on. If $g$ has
the Fourier expansion $g=\sum g_nz^n$,
$g_n=\left(\begin{matrix}a_n&b_n\\c_n&d_n\end{matrix}\right)$,
then relative to the basis for $\mathcal H$:
\begin{equation}\label{basis} ..
\epsilon_1z,\epsilon_2z,\epsilon_1,\epsilon_2,\epsilon_1z^{-1},\epsilon_2z^{-1},..\end{equation}
the matrix of $M_g$ is block periodic of the form
$$\begin{array}{ccccccccc}&.&.&.&.&.&.&.&\\..&a_0&b_0&a_1&b_1&\vert&a_2&b_2&..
\\..&c_0&d_0&c_1&d_1&\vert&c_{2}&d_{2}&..\\
..&a_{-1}&b_{-1}&a_0&b_0&\vert&a_1&b_1&.. \\
..&c_{-1}&d_{-1}&c_0&d_0&\vert&c_1&d_1&..\\-&-&-&-&-&-&-&-&-\\
..&a_{-2}&b_{-2}&a_{-1}&b_{-1}&\vert&a_0&b_0&..\\
..&c_{-2}&d_{-2}&c_{-1}&d_{-1}&\vert&c_0&d_0&..\\&.&.&.&.&.&.&.&
\end{array}$$
From this matrix form, it is clear that, up to equivalence, $M_g$
has just two types of ``principal minors", the matrix representing
$A(g)$, and the matrix representing the shifted Toeplitz operator
$A_1(g)$, the compression of $M_g$ to the subspace spanned by
$\{\epsilon_iz^j:i=1,2,j>0\}\cup\{\epsilon_1\}$. Relative to the
basis (\ref{basis}), the involution $\sigma$ defined by
(\ref{outer}) is equivalent to conjugation by the shift operator,
i.e. the matrix of $M_{\sigma(g)}$ is obtained from the matrix for
$M_g$ by shifting one unit along the diagonal (in either
direction: the result is the same, because $M_g$ commutes with
$M_z$, the square of the shift operator). Consequently the shifted
Toeplitz operator is equivalent to the operator $A(\sigma(g))$.

\begin{theorem}\label{factorization} Suppose that $g\in L^{\infty}(S^1,SL(2,\mathbb C))$.

(a) If $A(g)$ is invertible, then $g$ has a Birkhoff
factorization, where
\begin{equation}\label{factorformula}(g_0g_+)^{-1}=[A(g)^{-1}\left(\begin{matrix}1\\0\end{matrix}\right),
A(g)^{-1}\left(\begin{matrix}0\\1\end{matrix}\right)].\end{equation}

(b) If $A(g)$ and $A_1(g)$ are invertible, then $g$ has a
triangular factorization.

\end{theorem}

\begin{proof} For part (a), let $M$ denote the $2\times 2$ matrix
valued loop on the right hand side of (\ref{factorformula}). The
columns of this matrix are in $\mathcal H^+$. We must check that
$det(M)=1$ on $\Delta$. Because the entries of $M$ are in
$L^2(S^1)$, $det(M)\in L^1(S^1)$. Because $det(g)=1$ on $S^1$, and
$gM=1+O(z^{-1})$, $det(M)$ is holomorphic on $\Delta$, and on
$S^1$ equals a function which is  holomorphic in $\Delta^*$ and
equal to 1 at $\infty$. Consequently $det(M)$ has a holomorphic
extension to all of $\hat{\mathbb C}$, and hence must be
identically $1$. We can now take $g_0g_+=M^{-1}$. This will have
$L^2$ entries, because $M$ is unimodular.

For part (b), suppose that $g$ has Birkhoff factorization
$g=g_-g_0g_+$, and let
$g_0=\left(\begin{matrix}\alpha&\beta\\\gamma&\delta\end{matrix}\right)$.
The matrix representing $M_{g_0g_+}$ has the form
$$\begin{array}{ccccccccc}&.&.&.&.&.&.&.&\\..&\alpha&\beta&*&*&\vert&*&*&..
\\..&\gamma&\delta&*&*&\vert&*&*&..\\
..&0&0&\alpha&\beta&\vert&*&*&.. \\
..&0&0&\gamma&\delta&\vert&*&*&..\\-&-&-&-&-&-&-&-&-\\
..&0&0&0&0&\vert&\alpha&\beta&..\\
..&0&0&0&0&\vert&\gamma&\delta&..\\&.&.&.&.&.&.&.&
\end{array}$$
The matrix representing $M_{g_-}$ is unipotent and lower
triangular. Consequently $A_1(g)=A_1(g_-)A_1(g_0g_+)$, $A_1(g_-)$
is unipotent lower triangular, and $A_1(g)$ is invertible iff
$A_1(g_0g_+)$ is invertible iff $\alpha=(g_0)_{11}\ne 0$. This
implies Part (b).

\end{proof}

In Theorem \ref{factorization} we are assuming that $g$ is
bounded. It is not generally true that the factors $g_{\pm}$ are
bounded. Recall (see \cite{CG}) that a Banach $*$-algebra $\mathbb
A \subset L^{\infty}(S^1)$ is said to be decomposing if
$$\mathbb A = \mathbb A_+ \oplus \mathbb A_-,$$ i.e. $P_+:\mathbb A
\to \mathbb A_+$ is continuous. For example $C^s(S^1)$ is
decomposing, provided $s>0$ and nonintegral (see page 60 of
\cite{CG}), and $W^s$ is a decomposing algebra, provided $s>1/2$
(Note: $W^{1/2}$ is not an algebra).

\begin{corollary}\label{decomposing} Suppose that $g \in L^{\infty}(S^1,SL(2,\mathbb
C))$ belongs to a decomposing algebra $\mathbb A$ and has a
Birkhoff factorization. Then the factors $g_{\pm}$ belong to
$\mathbb A$.
\end{corollary}

This follows from the continuity of $P_+$ on $\mathbb A$ and the
formula in (a) of Theorem \ref{factorization}.

\begin{theorem}\label{hartmann}If $g \in L^{\infty}(S^1,SL(2,\mathbb
C))$, then $B(g)$ and $C(g)$ are compact operators if and only if
$g\in VMO$, the space of functions with vanishing mean
oscillation. If $g\in QC:=L^{\infty}\cap VMO$, then $A(g)$ and
$D(g)$ are Fredholm of index $0$.
\end{theorem}

The first statement is due to Hartmann, and the second to Douglas
(see pages 27 and 108 of \cite{Pe}, respectively).

\begin{remarks}\label{perspective} (a) In the context of Theorem \ref{factorization},
if $g$ has a Birkhoff factorization, then $A(g)$ is $1-1$: for if
$h\in \mathcal H_+$, then there is a Hardy decomposition of (not
necessarily $L^2$) $\mathbb C^2$ valued functions
$$ g_-^{-1}(M_g h)_+=g_0g_+h-g_-^{-1}(M_g h)_-;$$ thus if $A(g)h=0$, then
$h=0$. A Birkhoff factorization for bounded $g$ does not imply
$A(g)$ is invertible (see Theorem 5.1, page 109 of \cite{Pe}).

(b) For $g\in QC(S^1,SL(2,\mathbb C))$, the converse in (a) (and
also (b)) of Theorem \ref{factorization} holds, because the
Fredholm index of $A(g)$ vanishes. Moreover there is a notion of
generalized triangular factorization for all $g$ (see \cite{CG}
and chapter 8 of \cite{PS}).

(c) Theorem \ref{hartmann} implies that the Toeplitz operator
defines a holomorphic map
$$ QC(S^1,SL(2,\mathbb C))\to Fred(\mathcal H_+):g\to A(g).$$
There is a determinant line bundle $Det \to Fred(\mathcal H_+)$
with canonical section, $A \to det(A)$, which is nonvanishing
precisely when $A$ is invertible. In the notation of \cite{P},
$\sigma_0=det(A(\tilde g))$ is the pullback of the canonical
section, and $\sigma_1=det(A(\sigma(\tilde g)))$, viewed as
holomorphic functions of $\tilde g$ in the universal $\mathbb C^*$
extension of $QC(S^1,SL(2,\mathbb C))$. If $g$ has a triangular
factorization, then
\begin{equation}\label{diagonal}m(g)a(g)=
\left(\begin{matrix}\sigma_1/\sigma_0&0\\0&\sigma_0/\sigma_1\end{matrix}\right),\end{equation}
as the matrix manipulations above suggest (see (1.5)-(1.6) of
\cite{P}).

\end{remarks}

\section{Proof of Theorem \ref{SU(2)theorem1}, and Generalizations
to Other Function Spaces}\label{SU(2)case}

In the course of proving Theorem \ref{SU(2)theorem1}, we will also
prove the following

\begin{theorem}\label{SU(2)theorem1smooth} Suppose that $k_1 \in C^s(S^1,SU(2))$,
where $s>0$ and nonintegral. The following are equivalent:

($a_1$) $k_1$ is of the form
$$k_1(z)=\left(\begin{matrix} a(z)&b(z)\\
-b^*&a^*\end{matrix} \right),\quad z\in S^1,$$ where $a,b\in
H^0(\Delta)$ have $C^s$ boundary values, $a(0)>0$, and $a$ and $b$
do not simultaneously vanish at a point in $\Delta$.

($c_1$) $k_1$ has triangular factorization of the form
$$\left(\begin{matrix} 1&0\\
\sum_{j=0}^ny^*_jz^{-j}&1\end{matrix} \right)\left(\begin{matrix} a_1&0\\
0&a_1^{-1}\end{matrix} \right)\left(\begin{matrix} \alpha_1 (z)&\beta_1 (z)\\
\gamma_1 (z)&\delta_1 (z)\end{matrix} \right),$$ where the factors
have $C^s$ boundary values.

Similarly, the following are equivalent:

($a_2$) $k_2$ is of the form
$$k_2(z)=\left(\begin{matrix} d^{*}&-c^{*}\\
c(z)&d(z)\end{matrix} \right),\quad z\in S^1,$$ where $c,d\in
H^0(\Delta)$ have $C^s$ boundary values, $c(0)=0$, $d(0)>0$, and
$c$ and $d$ do not simultaneously vanish at a point in $\Delta$.

($c_2$) $k_2$ has triangular factorization of the form
$$\left(\begin{matrix} 1&\sum_{j=1}^nx^*_jz^{-j}\\
0&1\end{matrix} \right)\left(\begin{matrix} a_2&0\\
0&a_2^{-1}\end{matrix} \right)\left(\begin{matrix} \alpha_2 (z)&\beta_2 (z)\\
\gamma_2 (z)&\delta_2 (z)\end{matrix} \right), $$ where the
factors have $C^s$ boundary values.

\end{theorem}

\begin{remarks} (a) When $k_2\in L_{fin}SU(2)$, the determinant condition
$c^*c+dd^*=1$ can be interpreted as an equality of finite Laurent
expansions in $\mathbb C^*$. Together with $d(0)>0$, this implies
that $c$ and $d$ do not simultaneously vanish. Thus the added
hypotheses in $(a_i)$ of Theorem \ref{SU(2)theorem1smooth} are
superfluous in the finite case.

(b) The kind of example we have to avoid in the $C^{\infty}$ case
is
$$k_2=\left(\begin{matrix}d^*&0\\0&d\end{matrix}\right),\quad d=\frac{z-r}{rz-1}$$ where
$0<r<1$.

(c) The factorizations in $(b_i)$ of Theorem \ref{SU(2)theorem1}
are akin to nonabelian Fourier expansions. Consequently it is
highly unlikely that one can characterize the coefficients for
$C^s$ loops. For this purpose we consider a Sobolev completion at
the end of this section.

\end{remarks}

\begin{proof} As we remarked in the Introduction, the two sets of
conditions are intertwined by the outer involution $\sigma$. Also
it is evident that $(c_2)\implies(a_2)$: by multiplying the
matrices in $(c_2)$, we see that $c=a_2^{-1}\gamma_2$ and
$d=a_2^{-1}\delta_2$, and these cannot simultaneously vanish at a
point in $\Delta$. We will now prove, in reference to Theorem
\ref{SU(2)theorem1}, that
$(b_2)\implies(a_2)\implies(c_2)\implies(b_2)$. The second step
will also complete the proof of Theorem \ref{SU(2)theorem1smooth}.

It is straightforward to calculate that a loop as in ($b_2$) has
the matrix form in ($a_2$):

\begin{proposition}\label{gammadelta}The product in ($b_2$) equals
$$\left( \prod a(\zeta_i) \right) \left(\begin{matrix} \delta_2^*&-\gamma_2^*\\
\gamma_2&\delta_2\end{matrix} \right),$$ where
$$\gamma_2(z)=\sum_{n=1}^{\infty}\gamma_{2,n}z^n,$$
$$\gamma_{2,n}=\sum (-\bar{\zeta}_{i_1})\zeta_{j_1}...(-\bar{\zeta}_{
i_r})\zeta_{j_r}(-\bar{\zeta}_{i_{r+1}}),$$ the sum over
multiindices satisfying
$$0<i_1<j_1<..<j_r<i_{r+1},\quad\sum i_{*}-\sum j_{*}=n,$$
and
$$\delta_2(z)=1+\sum_{n=1}^{\infty}\delta_{2,n}z^n,$$
$$\delta_{2,n}=\sum\zeta_{i_1}(-\bar{\zeta}_{j_1})...\zeta_{i_r}(
-\bar{\zeta}_{j_r}),$$ the sum over multiindices satisfying
$$0<i_1<j_1<..<j_r,\quad\sum (j_{*}-i_{*})=n.$$
\end{proposition}

This is a straightforward induction, which we omit.

Now suppose that we are given a loop $k_2$ satisfying the
conditions in ($a_2$), with one exception: for later convenience,
we initially assume that $k_2$ is merely measureable. Suppose that
$$A(k_2)f=P_+(\left(\begin{matrix}d^*&-c^*\\c&d\end{matrix}\right)
\left(\begin{matrix}f_1\\f_2\end{matrix}\right))
=\left(\begin{matrix}0\\0\end{matrix}\right).$$ Then
$cf_1+df_2=0\in H^0(\Delta)$, and hence by the independence of $c$
and $d$ around $S^1$, $(f_1,f_2)=\lambda(d,-c)$. Because $c$ and
$d$ do not simultaneously vanish, this implies that $\lambda$ is
holomorphic in $\Delta$. We also have $(d^*\lambda
d-c^*\lambda(-c))_+=\lambda_+=0$. Thus $\lambda=0$. Thus the
Toeplitz operator is invertible [Note: conversely, if $c$ and $d$
have a common zero $z_0\in \Delta$, then the Toeplitz operator is
not invertible: take $\lambda=1/(z-z_0)$]. The same argument shows
that $A_1(k_2)$, and also $D(k_2)$, are invertible.

We must now show that this loop has a triangular factorization as
in ($c_2$), i.e. we must solve for $a_2$, $x^*$, and so on, in
\begin{equation}\label{solve}k_2(z)=\left(\begin{matrix} d^{*}&-c^{*}\\
c(z)&d(z)\end{matrix} \right)=\left(\begin{matrix} 1&\sum_{j=1}^n \bar x_jz^{-j}\\
0&1\end{matrix} \right)\left(\begin{matrix} a_2&0\\
0&a_2^{-1}\end{matrix} \right)\left(\begin{matrix} \alpha_2 (z)&\beta_2 (z)\\
\gamma_2 (z)&\delta_2 (z)\end{matrix} \right).\end{equation} The
form of the second row implies that we must have $a_2=d(0)^{-1}$ ,
and
\begin{equation}\label{eqn0}\gamma_2=a_2 c, \quad \text{and} \quad \delta_2=a_2 d.\end{equation}
because $\delta_2(0)=1$. This does define $a_2>0$, $\gamma_2$ and
$\delta_2$ in a way which is consistent with $(c_2)$, because
$c(0)=0$ and $d(0)>0$.

Using (\ref{eqn0}), the first row in (\ref{solve}) is equivalent
to
\begin{equation}\label{eqn1}d^{*}=\alpha_2+x^{*}c,\quad and\quad
-c^{*}=\beta_2+x^{*}d\end{equation} In the finite case, by
considering the second equation as an equality in $\mathbb C^*$,
we can immediately obtain that $x^*=-(c^*/d)_-$. The $C^s$ case is
more involved.

Consider the Hardy space polarization
$$H:=L^2(S^1,d\theta )=H^{+}\oplus H^{-},$$
and the operator
$$T:H^{-}\to H^{-}\oplus H^{-}:x^{*}\to (((cx^{*})_{-},(dx^{*})_{
-}).$$ The operator $T$ is the restriction of $D(k_2)^*=D(k_2^*)$
to the subspace $\{(x^*,0)\in \mathcal H^-\}$, consequently it is
injective with closed image.

The adjoint of $T$ is given by
$$T^{*}:H^{-}\oplus H^{-}\to H^{-}:(f^{*},g^{*})\to c^{*}f^{*}+d^{
*}g^{*}.$$ If $(f^{*},g^{*})\in ker(T^{*})$, then
$c^{*}f^{*}+d^{*}g^{*}$ vanishes in the closure of $\Delta^{*}$,
and because $\vert c\vert^2+\vert d\vert^ 2=1$ around $S^1$,
$(f^{*},g^{*})=\lambda^{*}(d^{*},-c^{*})$, where $\lambda^{*}$ is
holomorphic in $ \Delta^{*}$ and vanishes at $\infty$ because
$d^{*}(\infty )=d(0)>0$. We now claim that $(d^{*}_{-},-c^{*})\in
ker(T^{*})^{\perp}$: $$\int ((d^{*}_{-})f+(-c^{*})g)d\theta
=\int\lambda (d^{*}d+c^{*} c)d\theta =\int\lambda d\theta =0,$$
because $\lambda (0)=0$. Because $T$ has closed image, there
exists $x^{*}\in H^{-}$ such that
\begin{equation}\label{eqn2}d^{*}_{-}=(x^{*}c)_{-},\quad and\quad
-c^{*}=(x^{*}d)_{-}.\end{equation} We can now solve for $\alpha_2$
and $\beta_2$ in (\ref{eqn1}). This shows that $k_2$ in $(a_2)$
has a triangular factorization as in $(c_2)$. When $k_2\in C^s$,
by Corollary \ref{decomposing}, the factors are $C^s$. This
completes the proof of Theorem \ref{SU(2)theorem1smooth}.

We have now shown that $(b_2) \implies (a_2) \implies (c_2)$. To
prove that $(c_2)$ implies $(b_2)$, one method is to explicitly
solve for $x^*$ in terms of the $\zeta$ variables, then show that
this relation can be inverted. The formula for $x$ in terms of
$\zeta$ is discussed in the Appendix. For our present purposes we
only need to know that
$$x^{*}=\sum_{j=1}^{\infty}x_1^{*}(\zeta_j,..)z^{-j},$$
where
$$x_1^{*}(\zeta_1,..)=\zeta_1\prod_{k=2}^{\infty}
(1+\vert\zeta_k\vert^2)+\zeta_2\prod_{k=3}^{\infty}
(1+\vert\zeta_k\vert^2)s_2(\zeta_2,\zeta_3,..)+\zeta_3\prod_{k=4}^{\infty}
(1+\vert\zeta_k\vert^2)s_3(\zeta_3,\zeta_4,..)+.. $$ (in the
current context, these are finite sums). This structure implies
that we can solve for the $\zeta_j$ in terms of the $x_i$, and in
fact
$$\zeta_n(x_1,x_2,..)=\zeta_1(x_n,x_{n+1},..).$$ (Note: the
equivalence of $(b_2)$ and $(c_2)$ is implied by Theorem 5 of
\cite{P}, which uses Lie theory; here we are emphasizing the
elementary nature of the correspondence). This completes the proof
of Theorem \ref{SU(2)theorem1}.
\end{proof}

It is obvious that for $k_2$ in Theorem \ref{SU(2)theorem1}, there
is a factorization $a_2=\prod a(\zeta_j)^{-1}$. By considering the
Kac-Moody central extension of $LSU(2)$, one can obtain a
refinement of this factorization (recall (\ref{diagonal}), which
suggests the existence of this refinement).

\begin{theorem}\label{det1} For $k_i$ as in Theorem \ref{SU(2)theorem1},
$det(A^*A(k_1))$ equals
$$\lim_{N\to{\infty}}det(A_N(k_1))=det(1-C^*C(k_1))
=det(1+\dot B^*\dot B(y))^{-1}=\prod_{n\ge 1}(1+\vert \eta_n
\vert^2)^{-n}$$ and $det(A^*A(k_2))$ equals
$$\lim_{N\to{\infty}}det(A_N(k_2))=det(1-C^*C(k_2)) =det(1+\dot
B^*\dot B(x))^{-1}=\prod_{n\ge 1}(1+\vert \zeta_n\vert^2)^{-n},$$
where $A_N$ denotes the finite dimensional compression of $A$ to
the span of $\{\epsilon_i z^k:0\le k\le N\}$, and in the third
expressions, $x$ and $y$ are viewed as multiplication operators on
$H=L^2(S^1)$, with Hardy space polarization.
\end{theorem}

The first equalities are special cases of Theorem 6.1 of \cite{W};
these are included for perspective: they demonstrate that finite
dimensional approximations detect the magnitude of $detA$, not its
phase. The second equalities follow from the unitarity of the
$M_{k_i}$; they explain why the determinants are well-defined,
since $C(k_i)$ is Hilbert-Schmidt if and only if $k_i\in W^{1/2}$
(this follows immediately from the matrix expression for $M_{k_i}$
in Section \ref{triangularfactorization}). The last two equalities
follow from Theorem 5 of \cite{P}.

\begin{lemma}\label{weaklimit}Suppose that $\zeta=(\zeta_n)\in l^2$. As in Theorem
\ref{SU(2)theorem1}, let
$$k_2^{(N)}=\left(\begin{matrix} d^{(N)*}&-c^{(N)*}\\
c^{(N)}&d^{(N)}\end{matrix} \right):=
\left(\prod_{n=1}^{N}a(\zeta_n)\right)\left(\begin{matrix} 1&\zeta_Nz^{-N}\\
-\bar{\zeta}_Nz^N&1\end{matrix} \right)..\left(\begin{matrix} 1&\zeta_1z^{-1}\\
-\bar{\zeta}_1z&1\end{matrix} \right).$$ Then $c^{(N)}$ and
$d^{(N)}$ converge uniformly on compact subsets of $\Delta$ to
holomorphic functions $c=c(\zeta)$ and $d=d(\zeta)$, respectively,
as $N\to\infty$. The functions $c$ and $d$ have radial limits at
a.e. point of $S^1$, $c$ and $d$ are uniquely determined by these
radial limits,
$$k_2(\zeta):=\left(\begin{matrix} d(\zeta)^*&-c(\zeta)^*\\
c(\zeta)&d(\zeta)\end{matrix} \right)\in Meas(S^1,GL(2,\mathbb
C)),$$ and $det(k_2) \le 1$ on $S^1$.
\end{lemma}

A crucial lingering issue is the unitarity of $k_2$. In the course
of proving Theorem \ref{SU(2)theorem1W}, we will prove that $k_2$
is unitary on $S^1$ when $\zeta \in w^{1/2}$. Conjecturally this
is true for $\zeta \in l^2$ (see Section \ref{Summary}).

\begin{proof} Because $d^{(N)}d^{(N)*}+c^{(N)}c^{(N)*}=1,$ both $(c^{(N)})$ and
$(d^{(N)})$ are sequences of holomorphic functions on $\Delta$
which are bounded by $1$. By the Arzela-Ascoli Theorem, there
exist subsequences which converge uniformly to holomorphic
functions on $\Delta$, which will also be bounded by $1$.

We claim these limits are unique. As in Proposition
\ref{gammadelta}, write $k^{(N)}$ as
$$\left(\prod_{n=1}^{N}a(\zeta_n)\right)\left(\begin{matrix} \delta_2^{(N)*}&-\gamma_2^{(N)*}\\
\gamma_2^{(N)}&\delta_2^{(N)}\end{matrix} \right).$$ The
$\prod_{n=1}^{\infty}a(\zeta_n)$ converges, because $\zeta\in
l^2$. Proposition \ref{gammadelta} gives explicit expressions for
the coefficients of $\gamma_2^{(N)}$ and $\delta_2^{(N)}$. Very
crude estimates show that these expressions have well-defined
limits as $N\to \infty$. To see this, consider the formula for the
$n$th coefficient of $\delta_2$, and let $\mathcal P(n)$ denote
the set of partitions of $n$ (i.e. decreasing sequences $n_1\ge
n_2 \ge ..\ge n_l>0$, where $\sum n_j=n$ is the magnitude and
$l=l({n_j})$ is the length of the partition). Then
\begin{equation}\label{coefficient}\vert \delta_{2,n}\vert
\le\sum\vert \zeta_{i_1}\vert \vert \bar{\zeta}_{j_1}\vert
...\vert \zeta_{i_r}\vert \vert \bar{\zeta}_{j_r}\vert
,\end{equation} where the sum is over multiindices satisfying
$$0<i_1<j_1<..<j_r,\quad\sum (j_{*}-i_{*})=n.$$
If $n_k=j_k-i_k$, then $\sum n_k=n$, but this sequence is not
necessarily decreasing. However if we eliminate the constraints
$i_1<..<i_r$, then we can permute the indices ($1\le k\le r$) for
the $i_k$ and $n_k$. We can crudely estimate that
(\ref{coefficient}) is
$$\le \sum_{(n_i)\in \mathcal P (n)} \sum_{i_1,..,i_l>0}
\vert\zeta_{i_1}\vert
\vert\zeta_{i_1+n_1}\vert..\vert\zeta_{i_l}\vert
\vert\zeta_{i_l+n_l}\vert= \sum_{(n_i)\in \mathcal P (n)}
\prod_{s=1}^l \sum_{i_s>0}\vert\zeta_{i_s}\vert
\vert\zeta_{i_s+n_s}\vert$$
$$\le\sum_{\mathcal P(n)}\vert \zeta \vert_{l^2}^{2l((n_i))}$$
This shows that the Taylor coefficients of any limiting function
for the $\delta^{(N)}$ will be given by the formulas in
Proposition \ref{gammadelta}. The same considerations apply to the
$\gamma^{(N)}$. Thus the sequences $(\gamma^{(N)})$ and
$(\delta^{(N)})$ converge uniformly on compact sets of $\Delta$ to
unique limiting functions. This proves our claim about uniqueness
of the limits $c$ and $d$.

Because $c$ and $d$ are bounded by $1$ on $\Delta$, $c$ and $d$
have radial limits at a.e. point of $S^1$, and these boundary
values uniquely determine $c$ and $d$.

Finally we consider $det(k_2)$ on $S^1$. Since $c$ and $d$ are
holomorphic in $\Delta$, and $d(0)=\prod a(\zeta_j)\ne 0$,
$det(k_2)=\vert d\vert^2+\vert c\vert^2$ is nonzero a.e. on $S^1$.
Thus $k_2$ is invertible a.e. on $S^1$. Clearly $\vert
d\vert^2+\vert c\vert^2\le 2$ on the closure of $\Delta$, since
$\vert d\vert$ and $\vert c\vert$ are bounded by $1$. This also
holds for $d^{(N)}$ and $c^{(N)}$. If $\rho\in L^1(S^1,d\theta)$
is positive, then
$$\int_{S^1}(\vert d\vert^2+\vert c\vert^2)\rho d\theta
=\lim_{r\uparrow 1}\int_{S^1}(\vert d\vert^2+\vert
c\vert^2)(re^{i\theta})\rho(e^{i\theta}) d\theta,$$ (by dominated
convergence)
$$=\lim_{r\uparrow 1}\lim_{N\to\infty}\int_{S^1}(\vert
d^{(N)}\vert^2+\vert
c^{(N)}\vert^2)(re^{i\theta})\rho(e^{i\theta})
d\theta\le\lim_{N\to\infty}\limsup_{r\uparrow 1}\int_{S^1}(\vert
d^{(N)}\vert^2+\vert
c^{(N)}\vert^2)(re^{i\theta})\rho(e^{i\theta}) d\theta$$
$$=\lim_{N\to\infty}\int_{S^1}(\vert d^{(N)}\vert^2+\vert
c^{(N)}\vert^2)(e^{i\theta})\rho(e^{i\theta})
d\theta=\int_{S^1}\rho(e^{i\theta}) d\theta$$ Since $\rho$ is a
general positive integrable function, this implies that $\vert
d\vert^2+\vert c\vert^2\le 1$ on $S^1$.

This completes the proof.
\end{proof}

\begin{remark}To show that $k_2$ has values in $SU(2)$, it would suffice to show
\begin{equation}\label{goal}\frac1{2\pi}\int_{S^1}(\vert d\vert^2+\vert
c\vert^2)d\theta=1.\end{equation} This would follow immediately
(by dominated convergence) if we knew that $c^{(N)}$ ($d^{(N)}$)
converged to $c$ ($d$, respectively) on $S^1$. But we have not
shown this. Since $d(0)=\prod a(\zeta_j)$, it is clear that
(\ref{goal}) is bounded below by $\prod a(\zeta_j)^2$.
\end{remark}

\begin{theorem}\label{SU(2)theorem1W} Suppose that $k_1\in Meas(S^1,SU(2))$.
The following are equivalent:

($a_1$) $k_1$ is of the form
$$k_1(z)=\left(\begin{matrix} a(z)&b(z)\\
-b^*&a^*\end{matrix} \right),\quad z\in S^1,$$ where $a,b\in
H^0(\Delta)$ have $W^{1/2}$ boundary values, $a(0)>0$, and $a$ and
$b$ do not simultaneously vanish at a point in $\Delta$.

($b_1$) $k_1$ has a factorization of the form
$$k_1(z)=\lim_{n\to\infty}a(\eta_n)\left(\begin{matrix} 1&-\bar{\eta}_nz^n\\
\eta_nz^{-n}&1\end{matrix} \right)..a(\eta_0)\left(\begin{matrix}
1&
-\bar{\eta}_0\\
\eta_0&1\end{matrix} \right),$$ where $\eta\in w^{1/2}$, and the
limit is understood as in Lemma \ref{weaklimit}.

($c_1$) $k_1$ has triangular factorization of the form
$$\left(\begin{matrix} 1&0\\
\sum_{j=0}^ny^*_jz^{-j}&1\end{matrix} \right)\left(\begin{matrix} a_1&0\\
0&a_1^{-1}\end{matrix} \right)\left(\begin{matrix} \alpha_1 (z)&\beta_1 (z)\\
\gamma_1 (z)&\delta_1 (z)\end{matrix} \right),$$ where $y$ has
$W^{1/2}$ boundary values.

Moreover this defines a bijective correspondence between $\eta\in
w^{1/2}$ and $(y_n)\in w^{1/2}$.

Similarly, the following are equivalent:

($a_2$) $k_2$ is of the form
$$k_2(z)=\left(\begin{matrix} d^{*}&-c^{*}\\
c(z)&d(z)\end{matrix} \right),\quad z\in S^1,$$ where $c,d\in
H^0(\Delta)$ have $W^{1/2}$ boundary values, $c(0)=0$, $d(0)>0$,
and $c$ and $d$ do not simultaneously vanish at a point in
$\Delta$.

($b_2$) $k_2$ has a factorization of the form
$$k_2(z)=\lim_{n\to\infty}a(\zeta_n)\left(\begin{matrix} 1&\zeta_nz^{-n}\\
-\bar{\zeta}_nz^n&1\end{matrix}
\right)..a(\zeta_1)\left(\begin{matrix} 1&
\zeta_1z^{-1}\\
-\bar{\zeta}_1z&1\end{matrix} \right),$$ where $\zeta\in w^{1/2}$,
and the limit is understood as in Lemma \ref{weaklimit}.

($c_2$) $k_2$ has triangular factorization of the form
$$\left(\begin{matrix} 1&\sum_{j=1}^nx^*_jz^{-j}\\
0&1\end{matrix} \right)\left(\begin{matrix} a_2&0\\
0&a_2^{-1}\end{matrix} \right)\left(\begin{matrix} \alpha_2 (z)&\beta_2 (z)\\
\gamma_2 (z)&\delta_2 (z)\end{matrix} \right), $$ where $x$ has
$W^{1/2}$ boundary values.

Moreover this defines a bijective correspondence between $\zeta\in
w^{1/2}$ and $(x_n)\in w^{1/2}$.

\end{theorem}

\begin{remarks}\label{continuityremark}(a) If $\sum \vert \zeta_n \vert<\infty$, then
the products in $(b_i)$ converge absolutely and uniformly in $z\in
S^1$, and the limits are $C^0$. However $\sum n\vert
\zeta_n\vert^2<\infty$ does not imply absolute convergence of the
sum of the $\{\zeta_n\}$ and vice versa; similarly $C^0$ does not
imply $W^{1/2}$ and vice versa. It is for this reason that the
weak notion of convergence in Lemma \ref{weaklimit} is used in
$(b_i)$.

(b) In connection with $(b_i)$, note that $z^n$ converges to zero
uniformly on compact subsets of $\Delta$, but $\vert z^n \vert=1$,
for all $n$, on $S^1$. Thus it is not evident in $(b_i)$ that
$k_2$ is unitary; this is the problem which we could not resolve
in Lemma \ref{weaklimit}.
\end{remarks}

\begin{proof} The two sets of conditions are intertwined by $\sigma$.
We will first show $(a_2)$ is equivalent to $(c_2)$; we will then
show these conditions are equivalent to $(b_2)$.

Suppose that $k_2$ satisfies the conditions in $(a_2)$, except
that at the outset we only assume $k_2$ is measurable. In the
course of proving Theorem \ref{SU(2)theorem1smooth}, we showed
that $k_2$ has a triangular factorization as in $(c_2)$, where
\begin{equation}\label{ktox}\left(\begin{array}{c}
x^{*}\\
0\end{array} \right)=D(k_2^{*})^{-1}\left(\begin{array}{c}
(d^*)_{-}\\
-c^{*}\end{array} \right)\end{equation} (and the other factors are
given explicitly by (a) of Theorem \ref{factorization}). In
particular $x^*\in L^2$.

For the Birkhoff factorization of $k_2$,
$$(k_2)_-=\left(\begin{matrix}1&x^*\\0&1\end{matrix}\right).$$
Because $M_{k_2}$ is unitary,
\begin{equation}\label{append1}A(k_2)A(k_2)^*=(1+Z(k_2)^*Z(k_2))^{-1},\end{equation}
where $Z(k_2):=C(k_2)A(k_2)^{-1}$. A matrix calculation (see
(5.13) and (5.14) of \cite{P}, and note that in \cite{P}, $g=k_2$,
and $x$ is written in place of $x^*$) shows that
\begin{equation}\label{append2}Z(k_2)=Z((k_2)_-)=C((k_2)_-),\end{equation}
and relative to the basis (\ref{basis}), $C((k_2)_-)$ is
represented by the matrix
\begin{equation}\label{6.73}\left(\begin{matrix} .&0&x_n&.&0&x_3&0&x_2&0&x_1\\
.&0&0&.&0&0&0&0&0&0\\
.&.&.&&&x_4&0&x_3&0&x_2\\
&&&&&0&0&0&0&0\\
.&&&&&&&&0&x_3\\
.&&.&&&.&.\\
.&&&&&.&&&0&x_n\\
0&0&0&0&..&&&0&0&0\end{matrix} \right). \end{equation}

Now suppose that $k_2\in W^{1/2}$. In this case $A(k_2)A(k_2)^*$
is the identity plus trace class. By (\ref{append1}) and
(\ref{append2}), $C((k_2)_-)$ is Hilbert-Schmidt. By (\ref{6.73}),
$x^*\in W^{1/2}$.

Conversely, given $x^*\in W^{1/2}$, by Lemma 4 of \cite{P}, we can
explicitly compute $k_2$ and the corresponding triangular
factorization:
\begin{equation}\label{xtok1}\gamma_2
=-((1+\dot C(zx^{*})\dot C(zx^{*})^{*})^{-1}(x^{*}))^{*}, \qquad
\delta^{*}_2=1+\dot C(x^{*})\gamma_2\end{equation}
\begin{equation}\label{xtok2}\alpha_2=a_2^{-2}(1-\dot A(x^*)(\gamma_2)), \qquad
\beta=-a_2^{-2}\dot A(x^*)(\delta_2)\end{equation} and
\begin{equation}\label{xtok3}a_2^2=\frac{det(1+\dot C(x^*)^*\dot C(x^*))}{det(1+\dot C(zx^*)^*\dot C(zx^*))}\end{equation}
In the derivation of the equations (\ref{xtok1}) and (\ref{xtok2})
in Lemma 4 of \cite{P}, the fact that $k_2$ is unimodular is not
used explicitly; the derivation only uses
$(k_2)_{(1,1)}=(k_2)_{(2,2)}^*$ and
$(k_2)_{(1,2)}=-(k_2)_{(2,1)}^*$. However, because $\alpha_2
\delta_2-\beta_2 \gamma_2\in H(\Delta)$, and has real values
$\vert c \vert^2+\vert d \vert^2$ on $S^1$, $\alpha_2
\delta_2-\beta_2 \gamma_2$ extends holomorphically to
$\hat{\mathbb C}$. Since it equals $1$ at $z=0$, it is identically
$1$. This shows that unimodularity follows automatically. This
determines a unitary $k_2$ with measurable coefficients. The
calculations (\ref{append1}), (\ref{append2}), and (\ref{6.73})
imply that $k_2\in W^{1/2}$. Thus $(a_2)$ is equivalent to
$(c_2)$.

Lemma \ref{weaklimit} implies that if $(\zeta_n)\in l^2$, then
$k_2$ defined as in $(b_2)$ is in $Meas(S^1,GL(2,\mathbb C))$.
Now suppose that $\zeta\in w^{1/2}$. By Theorem \ref{det1}
\begin{equation}\label{det6}det \vert A(k_2^{(N)})\vert^2=det(1+\dot
B(x^{(N)})\dot B(x^{(N)})^*)^{-1} =\prod_{n=1}^N (1+\vert \zeta_n
\vert^2)^{-n},\end{equation} and this converges to a positive
number as $N\to \infty$.

First suppose that $\zeta\ge 0$. Proposition \ref{positivity} of
the Appendix implies that the coefficients of $x(\zeta)^{(N)}$ are
nonnegative and converge up to the coefficients of $x(\zeta)$.
This implies that the matrix entries of $B(x^{(N)}\dot
B(x^{(N)})^*$ will be nonnegative and  converge in a monotone way
to those for $\dot B(x)\dot B(x)^*$. Thus the sequence
$tr(B(x^{(N)}\dot B(x^{(N)})^*)$, which is bounded because
(\ref{det6}) converges, will converge to $tr(\dot B(x)\dot
B(x)^*)$. This implies that $(x_n)\in w^{1/2}$. For a general
$\zeta\in w^{1/2}$, since the coefficients for $x(\vert \zeta
\vert)$ dominate those for $x(\zeta)$ we can conclude in the same
way that $(x_n)\in w^{1/2}$. We can now obtain a triangular
factorization for $k_2$ using (\ref{xtok1})-(\ref{xtok3}). As we
argued in the paragraph following (\ref{xtok3}), this
automatically implies that $k_2$ is unitary. The calculations
(\ref{append1}), (\ref{append2}), and (\ref{6.73}) imply that
$k_2\in W^{1/2}$ and $A(k_2)$ is invertible. Since $A(k_2)$ is
$1-1$, this implies that $c$ and $d$ do not simultaneously vanish
in $\Delta$ (see the Note in the second paragraph following
Proposition \ref{gammadelta}). Thus $(b_2)$ implies $(a_2)$.

Suppose that we are given $k_2$ and $x$ as in $(a_2)$ and $(c_2)$.
Let $x^{(N)}=\sum_{n=1}^N x_n z^n$, and let $\zeta^{(N)}$ and
$k_2^{(N)}$ denote the corresponding objects. Theorem \ref{det1}
implies that
\begin{equation}\label{det3eqn}det(1+\dot{B}(x^{(N)})\dot{B}(x^{(N)})^*)
=\prod_{n=1}^N(1+\vert\zeta_n^{(N)}\vert^2)^n.\end{equation}
Because $x\in W^{1/2}$, the sequence of numbers (\ref{det3eqn})
has a limit. Therefore the sequence $\{\zeta^{(N)}\}$ is bounded
in $w^{1/2}$. Because the inclusion $w^{1/2}\to l^2$ is a compact
operator, there are subsequences which converge in $l^2$. By Lemma
\ref{weaklimit} these limiting sequences correspond to $k_2$. Thus
there is a unique limiting sequence, $\{\zeta_n\}\in l^2$. Since
(\ref{det3eqn}) has a limit, $\zeta\in w^{1/2}$. Thus $(a_2)$ and
$(c_2)$ imply $(b_2)$.

This completes the proof.

\end{proof}

\section{Proof of Theorem \ref{U(2)theorem}, and
Generalizations}\label{SU(2)caseII}

Part (a) of Theorem \ref{U(2)theorem} is obvious.  We will deduce
the remaining parts of Theorem \ref{U(2)theorem} from the
following

\begin{theorem}\label{maintheorem}Assume $s>0$ and nonintegral, or
$s=\infty$. For $g\in C^s(S^1,SU(2))$, the following are
equivalent:

(i) $g$ has a triangular factorization $g=lmau$, where $l$ and $u$
have $C^s$ boundary values.

(ii) $g$ has a factorization $g=k_1^*\lambda k_2$, where the
$k_i\in C^s(S^1,SU(2))$ satisfy the equivalent conditions $(a_i)$
and $(c_i)$ of Theorem \ref{SU(2)theorem1smooth}, and $\lambda\in
C^s(S^1,T)_0$.

\end{theorem}

\begin{proof}We will use the notation in (\ref{factorization}) for $g$,
and the notation in Theorem \ref{SU(2)theorem1smooth} for the
entries of the $k_i$ and their triangular factorizations. Without
much comment, we will use the fact that $C^s$ is a decomposing
algebra, so that factors in various decompositions will remain in
$C^s$.

We proved that (ii) implies (i) in \cite{P} (see the proof of
Theorem 7); we briefly recall the calculation. Suppose that $g\in
C^s(S^1,SU(2))$ can be factored as
$g=k_1^*\left(\begin{matrix}\lambda&0\\0&\lambda^{-1}
\end{matrix}\right) k_2$, as in (ii). We can write
$\lambda=exp(-\chi^*+\chi_0+\chi)$, where $\chi_0 \in i\mathbb R$
and $\chi\in H^0(\Delta)$, $\chi(0)=0$, with $C^s$ boundary
values. Then $g$ has triangular factorization of the form
\begin{equation}g=l(g)\left(\begin{matrix}e^{\chi_0}a_1a_2&0\\0&(e^{\chi_0}a_1a_2)^{-1}\end{matrix}\right)u(g),\end{equation}
where $m_0=e^{\chi_0}\in S^1$, $a_0=a_1a_2>0$,
\begin{equation}\label{lmatrix}l(g):=\left(\begin{matrix}l_{11}&l_{12}\\l_{21}&l_{22}\end{matrix}\right)
=\left(\begin{matrix} \alpha_1^{*}&\gamma_1^{*}\\
\beta_1^{*}&\delta_1^{*}\end{matrix}
\right)\left(\begin{matrix}e^{-\chi^{*}}&0\\0&e^{\chi^*}\end{matrix}\right)\left(\begin{matrix}
1&a_1^2e^{2\chi_0}P_-(ye^{2\chi^{*}}+x^*e^{2\chi})\\
0&1\end{matrix} \right)\end{equation} and
\begin{equation}\label{umatrix}u(g):=\left(\begin{matrix}u_{11}&u_{12}\\u_{21}&u_{22}\end{matrix}\right)
= \left(\begin{matrix}
1&a_2^{-2}e^{-2\chi_0}P_+(ye^{2\chi}+x^*e^{2\chi^*})\\0&1\end{matrix}\right)
\left(\begin{matrix}e^{\chi}&0\\0&e^{-\chi}\end{matrix}\right)
\left(\begin{matrix} \alpha_2&\beta_2\\
\gamma_2&\delta_2\end{matrix}\right)\end{equation} Thus (i) is
implied by (ii).

Now suppose that $g$ has triangular factorization $g=lmau$ as in
(i). We must solve for $k_1$, $\chi$, and $k_2$. The equation
(\ref{lmatrix}) implies
\begin{equation}\label{1}l_{11}=\alpha_1^*exp(-\chi^*),\quad l_{21}=\beta_1^*exp(-\chi^*)\end{equation}
and (\ref{umatrix}) implies
\begin{equation}\label{2}u_{21}=\gamma_2exp(-\chi),\quad u_{22}=\delta_2exp(-\chi)\end{equation}
The special forms of $k_1$ and $k_2$ imply that on $S^1$,
\begin{equation}\vert \alpha_1 \vert ^2+\vert \beta_1 \vert ^2=a_1^{-2}.\end{equation}
\begin{equation}\vert \delta_2 \vert ^2+\vert \gamma_2 \vert ^2=a_2^{2}.\end{equation}
Therefore on $S^1$
\begin{equation}\vert l_{11} \vert ^2+ \vert l_{21} \vert ^2= a_1^{-2} exp(-2Re(\chi))\end{equation}
\begin{equation}\vert u_{21} \vert ^2+\vert u_{22} \vert^2=a_2^2exp(-2Re(\chi))\end{equation}
This implies that on $S^1$ we must have
\begin{equation}Re(\chi)=log(a_1^{-1})+log((\vert l_{11} \vert ^2+ \vert l_{21} \vert
^2)^{-1/2})=log(a_2)+log((\vert u_{21} \vert ^2+\vert u_{22} \vert
^2)^{-1/2}).\end{equation} Assuming that the obvious consistency
condition is satisfied, this pair of equations determines $\chi$
and the $a_i$: because $\chi$ must be holomorphic in the disk and
vanish at $z=0$, the average of $Re(\chi)$ around $S^1$ must
vanish, hence
\begin{equation}a_1=exp(-\frac1{4\pi}\int_{S^1}log(\vert l_{11} \vert
^2+ \vert l_{21} \vert ^2)d\theta),\end{equation}
\begin{equation}
a_2=exp(\frac1{4\pi}\int_{S^1}log(\vert u_{21} \vert ^2+\vert
u_{22} \vert ^2)d\theta),\end{equation} and
\begin{equation}Im(\chi)=iRe(\chi)_{-}-iRe(\chi)_+.\end{equation}

To see that $\chi$ and the $a_i$ are well-defined, we must check
that
\begin{equation}\label{consistency}\vert l_{11} \vert ^2+ \vert l_{21} \vert ^2=
(a_1a_2)^{-2}(\vert u_{21} \vert ^2+\vert u_{22}
\vert^2),\end{equation} as functions on $S^1$. Because $g^*g=1$,
$l^*l=(a(g)u)^{-*}(a(g)u)^{-1}$, on $S^1$. This implies three
independent equations

\begin{equation}\label{1,1}\vert l_{11}\vert^2+\vert l_{21}\vert^2=a_0^{-2}(\vert
u_{22}\vert^ 2+\vert u_{21}\vert^2)\end{equation}
\begin{equation}\label{2,1}l_{11}^{*}l_{12}+l_{21}^{*}l_{22}=-(u_{22}^{*}u_{12}+u_{21}^{*}
u_{11})\end{equation}
\begin{equation}\label{2,2}\vert
l_{12}\vert^2+\vert l_{22}\vert^2=a_0^2(\vert u_{12}\vert^2 +\vert
u_{11}\vert^2)\end{equation} for the $(1,1)$, $(1,2)$ (or
$(2,1)$), and $(2,2)$ entries, respectively. The $(1,1)$ entry
implies the consistency condition (\ref{consistency}).

Together with (\ref{1}) and (\ref{2}), this completely determines
the $k_i$: \begin{equation}a(z)=a_1exp(\chi)l_{11}^*,\quad
b(z)=a_1exp(\chi)l_{21}^*,\end{equation}
\begin{equation}
c(z)=a_2^{-1}exp(\chi)u_{21},\quad d(z)=a_2^{-1}exp(\chi)u_{22}
\end{equation} Because $l^*$ is invertible at all points of
$\Delta$, the entries $a$ and $b$ of $k_1$ do not simultaneously
vanish. Similarly, because $u$ is invertible, the entries $c$ and
$d$ do not simultaneously vanish. The fact that these are $C^s$ in
the appropriate sense follows from the continuity of the
projections $P_{\pm}$ on $C^s$. Thus by Theorem
\ref{SU(2)theorem1smooth} (and the ensuing Remark (b)) the $k_i$
have appropriate triangular factorizations.

We have now solved for $k_i$ and $\chi$. We have also observed
that the diagonal term of $g$ determines $exp(\chi_0)$, so
$\lambda$ is determined as well.

We now must show that $g=k_1^{-1}\lambda k_2$. From the
definitions of $k_i$ and $\lambda$, both sides of this equation
have the same $m$, $a$, $l_{11}$, $l_{21}$, $ u_{21}$, and
$u_{22}$ coordinates. The proof is completed by the next
Proposition, which is of intrinsic interest.
\end{proof}

\begin{proposition}\ Suppose that $g$ has a triangular
factorization as in (\ref{factorization}). Then
\[l_{12}=-l_{11}P_{-}(\frac {l_{21}^{*}+u_{21}^{*}}{\vert l_{11}\vert^
2+\vert l_{21}\vert^2})\]
\[l_{22}=1-l_{21}P_{-}(\frac {l_{21}^{*}+u_{21}^{*}}{\vert l_{11}
\vert^2+\vert l_{21}\vert^2})\]
\[u_{12}=-a_0^{-2}u_{22}P_{+}(\frac {l_{21}^{*}+u_{21}^{*}}{\vert
l_{11}\vert^2+\vert l_{21}\vert^2})\]
\[u_{11}=1-a_0^{-2}u_{21}P_{+}(\frac {l_{21}^{*}+u_{21}^{*}}{\vert
l_{11}\vert^2+\vert l_{21}\vert^2})\] In particular $g$ is
determined by $m$, $a$, $l_{11}$, $l_{21}$, $ u_{21}$, and
$u_{22}$.
\end{proposition}

\begin{proof} We initially suppose that $l_{11}$ and $u_{22}$ are
nonvanishing. We can use the unimodularity of $l$ and $u$ to solve
for $l_{22}$ and $u_{11}$ in terms of $l_{12}$ and $u_{12}$.

The equation (\ref{2,1}) can be rewritten as
\[l_{11}^{*}l_{12}+l_{21}^{*}l_{22}+u_{22}^{*}u_{12}+u_{21}^{*}u_{
11}=\]
\[l_{11}^{*}l_{12}+l_{21}^{*}(1+\frac {l_{12}l_{21}}{l_{11}})+u_{
22}^{*}u_{12}+u_{21}^{*}(1+\frac {u_{12}u_{21}}{u_{22}})=0\] Using
(\ref{1,1}) this can be rewritten as
\[\frac {l_{12}}{l_{11}}+a^2_0\frac {u_{12}}{u_{22}}=-\frac {l_{2
1}^{*}+u_{21}^{*}}{\vert l_{11}\vert^2+\vert l_{21}\vert^2}\] by
applying $P_{\pm}$ to this equation, and solving, we obtain the
equations in the proposition.
\end{proof}

Suppose that $g\in C^s(S^1,SU(2))$, $s>1/2$, and $g$ has a
triangular factorization. By Theorem 7 of \cite{P},
\begin{equation}\label{det2}det(A^*A(g))=det(A^*A(k_1^{-1}))det(A^*A(\lambda))det(A^*A(k_2))\end{equation}
$$=\prod_{i=1}^{\infty}(1+\vert\eta_i\vert^2)^{-i}exp(-2\sum_{j=1}^{\infty}j\vert
\chi_j\vert^2)\prod_{k=1}^{\infty}(1+\vert \zeta_k
\vert^2)^{-k}.$$ These expressions make sense because $C^s \subset
W^{1/2}$ for $s>1/2$. In the remainder of this section, our goal
is to use these equalities to obtain a $W^{1/2}$ analogue of
Theorem \ref{maintheorem}, which also incorporates the condition
$(b_i)$. This involves some subtleties, because $W^{1/2}$
functions are not necessarily continuous.

Because $SU(2)$ is compact, $W^{1/2}(S^1,SU(2))$ is a separable
topological group. In contrast to the function spaces $C^s$,
$s>0$, $W^s$, $s>1/2$, and $L^{\infty}\cap W^{1/2}$, for the
function space $W^{1/2}$, the loop group $W^{1/2}(S^1,SU(2))$ is
not a Lie group, because $W^{1/2}(S^1,su(2))$ is not a Lie algebra
(whereas, e.g. $L^{\infty}\cap W^{1/2}(S^1,su(2))$ has a Lie
algebra structure). Moreover the inclusion
$C^{\infty}(S^1,SU(2))\subset W^{1/2}(S^1,SU(2))$ is dense and
presumably a homotopy equivalence (whereas this is false for the
$L^{\infty}\cap W^{1/2}$ topology). With respect to the $W^{1/2}$
topology, the operator-valued function
$$g \to\left(\begin{matrix}A(g)&B(g)\\C(g)&D(g)\end{matrix}\right)$$ is
continuous, provided the diagonal is equipped with the strong
operator topology, and the off-diagonal with the Hilbert-Schmidt
topology.

In reference to the following Lemma, we recall that the notion of
degree (or winding number) can be extended from $C^0$ to
$VMO(S^1,S^1)$, hence degree is well-defined for
$W^{1/2}(S^1,S^1)$ (see Section 3 of \cite{Br} for an amazing
variety of formulas, and further references, or pages 98-100 of
\cite{Pe}). Also given $\lambda \in W^{1/2}(S^1,S^1)$, we view
$\lambda$ as a multiplication operator on $H=L^2(S^1)$, with the
Hardy polarization. We write $\dot A(\lambda)$ for the Toeplitz
operator, and so on (with the dot), to avoid confusion with the
matrix case.

\begin{lemma}\label{abelian} There is an exact sequence of topological groups
$$0 \to 2\pi i\mathbb Z \to W^{1/2}(S^1,i\mathbb R)
\stackrel{exp}{\rightarrow}
W^{1/2}(S^1,S^1)\stackrel{degree}{\rightarrow} \mathbb Z \to 0.$$
Moreover $degree(\lambda)=-index(\dot A(\lambda))$.
\end{lemma}

There is a more general version of this involving $VMO$, which is
implicit on pages 100-101 of \cite{Pe}.

\begin{proof}Suppose that $f\in W^{1/2}(S^1,i\mathbb R)$. It is convenient to use
the equivalent Besov form of the $W^{1/2}$ norm,
$$\vert f \vert^2_{W^{1/2}}=\int \int \frac{\vert f(\theta_1)-f(\theta_2)\vert^2}{\vert
e^{i\theta_1}-e^{i\theta_2}\vert^2}d\theta_1 d\theta_2.$$ Because
$\vert e^{i\theta}-1\vert\le\vert\theta\vert$,
$$\int \int
\frac{\vert e^{f(\theta_1)}-e^{f(\theta_2)}\vert^2}{\vert
e^{i\theta_1}-e^{i\theta_2}\vert^2}d\theta_1 d\theta_2\le\vert
f\vert^2_{W^{1/2}}.$$ Thus $exp(f)$ is also $W^{1/2}$. This
inequality also shows that $exp$ is continuous at $0$. Since $exp$
is a homomorphism, this implies $exp$ is globally continuous.

Continuity implies that the image of $exp$ is contained in the
identity component. Conversely suppose that $\lambda \in
W^{1/2}(S^1,S^1)_0$. Then $\dot A(\lambda)$ is invertible. This
implies the existence of a Birkhoff factorization
$\lambda=\lambda_-\lambda_0\lambda_+$, where for example
$\lambda_+\in H^0(\Delta,0;\mathbb C^*,1)$ and has $L^2$ boundary
values. By taking logarithms on the disks, we can write
$\lambda=exp(-\chi^*+\chi_0+\chi)$. By a formula of Szego and
Widom (Theorem 7.1 of \cite{W}),
\begin{equation}\label{det3}det(\dot A^*\dot A(\lambda))=det(1-\dot C^*\dot C(\lambda))
=exp(-2\sum_{j=1}^{\infty}j\vert \chi_j\vert^2)\end{equation} The
determinant depends continuously on $\lambda$ in the $W^{1/2}$
topology. Therefore $\chi \in W^{1/2}$. This shows the sequence is
exact at $W^{1/2}(S^1,S^1)$.

A $W^{1/2}$ function cannot have jump discontinuities. This
implies that the kernel of $exp$ is $2 \pi i \mathbb Z$. Thus the
sequence in the statement of the Lemma is continuous and exact.
\end{proof}

\begin{theorem}\label{maintheoremW}For $g\in W^{1/2}(S^1,SU(2))$, the following are
equivalent:

(i) $g$ has a triangular factorization $g=lmau$.

(ii) $g$ has a factorization $g=k_1^*\lambda k_2$, where the
$k_i\in W^{1/2}(S^1,SU(2))$ satisfy the equivalent conditions of
Theorem \ref{SU(2)theorem1W}, and $\lambda\in W^{1/2}(S^1,T)_0$.

In both cases the factorization is unique.
\end{theorem}

\begin{proof} Given Lemma \ref{abelian}, the proof that (ii) implies (i) is
the same as in the proof of Theorem \ref{maintheorem}.

Now assume (i). We can again solve for $k_i$ and $\chi$, as in the
proof of Theorem \ref{maintheorem}. The determinant formulas
(\ref{det2}) can be applied to
$g^{(N)}=k_1^{(N)}exp(\chi^{(N)})k_2^{(N)}$, where the subscript
indicates that $\zeta_n$, $\chi_n$, $\eta_n$ are set equal to $0$,
for $n>N$. In (\ref{det2}), applied to $g^{(N)}$, all of the
individual factors in (\ref{det2}) are bounded above by $1$, and
are tending monotonically down. Since $g\in W^{1/2}$,
$det(A(g)A(g)^*)$ is positive, and $det(A(g^{(N)})A(g{(N)})^*$
will remain bounded away from zero. This implies that all of the
factors in (\ref{det2}), applied to $g{(N)}$, will be bounded away
from $0$. Thus $\zeta$, $\chi$ and $\eta$ are in $w^{1/2}$. By
Theorem \ref{SU(2)theorem1W}, $k_i \in W^{1/2}$. This implies
(ii).

\end{proof}

\begin{corollary}\label{surprise} The dense open set of $g\in W^{1/2}(S^1,SU(2))$
having triangular factorization is parameterized by $y$,
$\chi_0\in i\mathbb R$ $mod 2\pi i \mathbb Z$, $\chi$, and $x$,
where $y$, $\chi$ and $x$ are holomorphic functions in $\Delta$
with $W^{1/2}$ boundary values, and $x(0)=\chi(0)=0$.

\end{corollary}

\begin{remark}This implies that an open neighborhood of $1\in W^{1/2}(S^1,SU(2))$ is
parameterized by a Hilbert space. This should be compared to the
finite dimensional situation, where a topological group locally
homeomorphic to $\mathbb R^n$ is automatically a $C^{\omega}$ Lie
group.
\end{remark}

\section{A Conjectural $L^2$ Generalization}\label{Summary}

Suppose that $\zeta \in l^2$. By Lemma \ref{weaklimit} there is a
unique limit $k_2\in Meas(S^1,GL(2,\mathbb C))$ for the product in
(\ref{append3}) below. When $A(k_2)$ is invertible, e.g. if $\zeta
\in w^{1/2}$ (by Theorem \ref{SU(2)theorem1W}), there are three
different expressions for $k_2$,

\begin{equation}\label{append3}\prod^{\leftarrow}a(\zeta_n)\left(\begin{matrix} 1&\zeta_nz^{-n}\\
-\bar{\zeta}_nz^n&1\end{matrix} \right)
=\left(\prod a(\zeta_n)\right)\left(\begin{matrix} \delta_2^{*}&-\gamma_2^{*}\\
\gamma_2&\delta_2\end{matrix} \right)=\left(\begin{matrix} 1&x^{*}\\
0&1\end{matrix} \right)\left(\begin{matrix} a_2&0\\
0&a_2^{-1}\end{matrix} \right)\left(\begin{matrix} \alpha_2&\beta_2\\
\gamma_2&\delta_2\end{matrix} \right),\end{equation} where
$a_2=\prod a(\zeta_j)^{-1}$, and $\gamma_2$ and $\delta_2$ are
determined by the formulas in Proposition \ref{gammadelta}. The
existence of the triangular factorization implies that $k_2$ has
values in $SU(2)$ on $S^1$.

Since the expression for $a_2$ is convergent for all $\zeta \in
l^2$, it is plausible that the triangular factorization in
(\ref{append3}) is valid for all $\zeta \in l^2$. A further leap
of faith suggests the following

\begin{conjecture}\label{SU(2)theorem1conjecture} Suppose that
$k_2\in Meas(S^1,SU(2))$. The following are equivalent:

($a_2$) $k_2$ is of the form
$$k_2(z)=\left(\begin{matrix} d^{*}&-c^{*}\\
c(z)&d(z)\end{matrix} \right),\quad z\in S^1,$$ where $c,d\in
H^0(\Delta)$, $c(0)=0$, $d(0)>0$, and $c$ and $d$ do not
simultaneously vanish at a point in $\Delta$.

($b_2$) $k_2$ has a factorization of the form
$$k_2(z)=\lim_{n\to\infty}a(\zeta_n)\left(\begin{matrix} 1&\zeta_nz^{-n}\\
-\bar{\zeta}_nz^n&1\end{matrix}
\right)..a(\zeta_1)\left(\begin{matrix} 1&
\zeta_1z^{-1}\\
-\bar{\zeta}_1z&1\end{matrix} \right),$$ where $\zeta\in l^2$, and
the limit is understood as in Lemma \ref{weaklimit}.

($c_2$) $k_2$ has triangular factorization of the form
$$\left(\begin{matrix} 1&\sum_{j=1}^nx^*_jz^{-j}\\
0&1\end{matrix} \right)\left(\begin{matrix} a_2&0\\
0&a_2^{-1}\end{matrix} \right)\left(\begin{matrix} \alpha_2 (z)&\beta_2 (z)\\
\gamma_2 (z)&\delta_2 (z)\end{matrix} \right).$$

Moreover this defines a bijective correspondence between $\zeta\in
l^2$ and $(x_n)\in l^2$.

\end{conjecture}

In reference to this Conjecture, recall that the condition $(a_2)$
implies that $A(k_2)$ is $1-1$. This entails invertibility when
$k_2\in QC$ (see Theorem \ref{hartmann}), but not in general. When
$k_2$ is expressed as in $(c_2)$, the third paragraph of the proof
of Theorem \ref{SU(2)theorem1W}, together with results of Nehari
and Fefferman (pages 3-5 of \cite{Pe}), implies that $A(k_2)$ is
invertible precisely when $x$ has $BMO$ boundary values. Thus the
implications $(b_2) \implies (a_2) \implies (c_2)$ hinge on the
question of whether $\zeta \in l^2 \implies (x_n)\in l^2$, and
this is different from the question of when $A(k_2)$ is
invertible.

The implication $(c_2) \implies (a_2)$ hinges on the formulas
(\ref{xtok1})-(\ref{xtok3}) for $k_2$ in terms of $x$. The first
two formulas make sense for $x\in BMO$, as in the preceding
paragraph, but it is not clear that this is the natural domain for
$x$. Regarding the formula for $a_2$, which a priori depends on
$(x_n)\in w^{1/2}$, the second order term in the expansion at
$x=0$ is
$$ tr(C(x^*)C(x^*)^*)-tr(C(zx^*)C(zx^*)^*)=\sum \vert x_n\vert^2,$$
the $l^2$ norm. This is at least consistent with the Conjecture.

\section{Appendix. The Relation Between $x^*$ and $\zeta$}\label{appendix}

In this Appendix, we consider the relation between $x^*$ and
$(\zeta_j)$, in Theorem \ref{SU(2)theorem1}, at the level of
combinatorial formulas.

\subsection{$x^*$ as a function of $\zeta$}

\begin{proposition}\label{positivity} $x^{*}$ has the form
$$x^{*}=\sum_{j=1}^{\infty}x_1^{*}(\zeta_j,..)z^{-j},$$
where
$$x_1^{*}(\zeta_1,..)=\sum_{n=1}^{\infty}\zeta_n\left(\prod_{k=n+1}^{\infty}
(1+\vert\zeta_k\vert^2)\right)s_n(\zeta_n,\zeta_{n+1},\bar{\zeta}_{n+1},
..),$$ $s_1=1$ and for $n>1$,
$$s_n=\sum_{r=1}^{n-1}s_{n,r},\quad s_{n,r}=\sum c_{i,j}\zeta_{i_1}\bar{\zeta}_{j_1}\zeta_{
i_2}\bar{\zeta}_{j_2}..\zeta_{i_r}\bar{\zeta}_{j_r}$$ where the
sum is over multiindices satisfying the constraints
\begin{equation}\label{index}\begin{matrix} &&j_1&\le&..&\le j_r&\\
&&\lor&&&\lor\\
n&\le&i_1&\le&..&i_r&\end{matrix}
,\quad\sum_{l=1}^r(j_l-i_l)=n-1,\end{equation} and $c_{i,j}$ is a
positive integer.
\end{proposition}

\begin{remark} The main features of the formula for $x^*_1$ are (i)
the appearance of the infinite products, which isolates the part
of the expression which has to be "renormalized" in probabilistic
applications, and (ii) the positivity of the coefficients. For
example (ii) implies that if $\zeta\ge 0$, then the coefficients
for $x(\zeta_1,..,\zeta_N,0,..)$ converge monotonically up to the
coefficients for $x(\zeta)$ as $N\to\infty$.
\end{remark}

\begin{proof} The fact that $x^*$ is completely determined by its residue $x^*_1$ is
(b) of Theorem 5 of \cite{P}. We will show that $x^*_1$ has the
form claimed in the Lemma (I stated this without proof in
\cite{P}).

Clearly $x_1^{*}(\zeta_1)=\zeta_1$. The proof hinges on the
following recursion (see Lemma 2 and (5.12) of \cite{P})
$$x^*_1(\zeta_1,..\zeta_{N+1})=$$
$$(1+\vert \zeta_{N+1}\vert^2)\{x_1(\zeta_1,..,\zeta_N)+\sum_{i+j=N+2}
x_1(\zeta_i,..,\zeta_N)x_1(\zeta_j,..,\zeta_N))\bar{\zeta}_{N+1}$$
$$+\sum_{i+j+k=2N+3}x_1(\zeta_i,..,\zeta_N)x_1(\zeta_j,..,\zeta_N
)x_1(\zeta_k,..,\zeta_N))\bar{\zeta}_{N+1}^2$$
$$+\sum_{i+j+k+l=3N+4}x_1(\zeta_i,..,\zeta_N)x_1(\zeta_j,..,\zeta_
N)x_1(\zeta_k,..,\zeta_N)x_1(\zeta_l,..\zeta_N)\bar{\zeta}_{N+1}^
3+..\}$$ From this recursion one can immediately see that
coefficients will be nonnegative.

We assume that
$$x_1^{*}(\zeta_1,..,\zeta_N)=\sum_{n=1}^N\zeta_n\prod_{k=n+1}^N(
1+\vert\zeta_k\vert^2)s_n(\zeta_n,..,\zeta_N),$$ where $s_1=1$ and
for $n>1$
$$s_n(\zeta_n,..,\zeta_N)=\sum c_{i,j}\zeta_{i_1}\bar{
\zeta}_{j_1}\zeta_{i_2}\bar{\zeta}_{j_2}..\zeta_{i_r}\bar{\zeta}_{
j_r},$$ the sum is over multiindices as in (\ref{index}), with
$j_r\le N$, and $c_{i,j}$ is a positive integer (for $N>1$, $s_
N(\zeta_N)=0$).

This implies
$$x_1^{*}(\zeta_I,..,\zeta_N)=\sum_{n=1}^{N-(I-1)}\zeta_{n+(I-1)}
\prod_{k=n+1}^{N-(I-1)}(1+\vert\zeta_{k+(I-1)}\vert^2)s_n(\zeta_{
n+(I-1)},..)$$
$$=\sum_{m=I}^N\zeta_m\prod_{k=m+1}^N(1+\vert\zeta_k\vert^2)s_{m-
(I-1)}(\zeta_m,..,\zeta_N)$$ where
$$s_{m-(I-1)}(\zeta_m,..,\zeta_N)=\sum c_{i-(I-1)\vec {1},
j-(I-1)\vec {1}}\zeta_{i_1}\bar{\zeta}_{j_1}..\zeta_{i_L}\bar{
\zeta}_{j_L},$$ the sum is over multiindices satisfying
$$\begin{matrix} &&j_1&\le&..&\le j_L&\le&N\\
&&\lor&&&\lor\\
m&\le&i_1&\le&..&i_L&\end{matrix}
,\quad\sum_{l=1}^L(j_l-i_l)=m-I,$$ and in the notation for the
coefficient, $i-(I-1)\vec {1}$ means that we subtract $I-1$ from
each of the components of $i$.

We now plug this into the recursion relation, and rewrite the
expression so that it has the same form as the sum involving $N$
variables:
\begin{equation}\label{sum1}x_1(\zeta_1,..\zeta_{N+1})=(1+\vert\zeta_{N+1}\vert^2)\sum_{s\ge
0}\{\sum_{\sum_{l=1}^{s+1}
I_l=s(N+1)+1}\prod_{I_l}x_1(\zeta_{I_l},..,\zeta_N)\}\bar{\zeta}_{
N+1}^s\end{equation}
$$=(1+\vert\zeta_{N+1}\vert^2)\sum_{s\ge 0}\sum_{\sum I_l=s(N+1)+
1}\prod_{I_l}(\sum_{m_l=I_l}^N\zeta_{m_l}\prod_{k=m_l+1}^N(1+\vert
\zeta_k\vert^2)s_{m_l-(I_l-1)}(\zeta_{m_l},..,\zeta_N))\bar{\zeta}_{
N+1}^s$$
$$=(1+\vert\zeta_{N+1}\vert^2)\sum_{s\ge 0}\sum_{\sum I_l=s(N+1)+
1}\sum_{m_1=I_1}^N..\sum_{m_{s+1}=I_{s+1}}^N$$
$$\prod_{I_l}[\zeta_{m_l}\prod_{k=m_l+1}^N(1+\vert\zeta_k\vert^2)
\sum c_{\vec i_l-(I_l-1)\vec {1},\vec {j}_l-(I_l-1)\vec
{1}}\zeta_{i_{l,1}}\bar{\zeta}_{j_{l,1}}..\zeta_{i_{l,L_l}}\bar{\zeta}_{j_{l
,L_l}}]\bar{\zeta}_{N+1}^s,$$
$$=(1+\vert\zeta_{N+1}\vert^2)\sum_{s\ge 0}\sum_{\sum I_l=s(N+1)+
1}\sum_{m_1=I_1}^N..\sum_{m_{s+1}=I_{s+1}}^N$$
$$\sum_1..\sum_{s+1}\prod_{I_l}[\zeta_{m_l}\prod_{k=m_l+1}^N(1+\vert
\zeta_k\vert^2)c_{\vec {i}_l-(I_l-1)\vec {1},\vec
{j}_l-(I_l-1)\vec {1}}\zeta_{
i_{l,1}}\bar{\zeta}_{j_{l,1}}..\zeta_{i_{l,L_l}}\bar{\zeta}_{j_{l
,L_l}}]\bar{\zeta}_{N+1}^s,$$ where for each $1\le l\le s+1$, the
sum $\sum_l$ is over multiindices satisfying
$$\begin{matrix} &&j_{l,1}&\le&..&\le j_{l,L_l}&\le&N\\
&&\lor&&&\lor\\
m_l&\le&i_{l,1}&\le&..&i_{l,L_l}&\end{matrix} ,\quad\sum_{\tau
=1}^{ L_l}(j_{l,\tau}-i_{l,\tau})=m_l-I_l,$$ Consider a term in
this sum of the form
\begin{equation}\label{term}\prod_{I_l}[\zeta_{m_l}\prod_{k=m_l+1}^N(1+\vert\zeta_k\vert^2)
\zeta_{i_{l,1}}\bar{\zeta}_{j_{l,1}}..\zeta_{i_{l,L_l}}\bar{\zeta}_{
j_{l,L_l}}]\bar{\zeta}_{N+1}^s,\end{equation} where $m_l\le
i_{l,1}$ for each $l$. Let $n=min\{m_l:1\le l\le s+1\}$, and
factor out
$$\zeta_n\prod_{k=n+1}^{N}(1+\vert\zeta_k\vert^2)$$ in
(\ref{term}). What remains can be expressed as a positive integral
combination of monomials
$$\zeta_{i_1}\bar{\zeta}_{j_1}\zeta_{i_2}\bar{\zeta}_{j_2}..\zeta_{
i_r}\bar{\zeta}_{j_L},$$ where
$$\begin{matrix} &&j_1&\le&..&\le j_L&\le&N+1\\
&&\lor&&&\lor\\
n&\le&i_1&\le&..&i_L&\end{matrix}
,\quad\sum_{l=1}^L(j_l-i_l)=n-1.$$ Multiplicities arise when the
factors with $m_l\ne m$,
$$\prod_{k=m_l+1}^N(1+\vert\zeta_k\vert^2)$$ are expanded.
Thus the entire sum can be written as
$$\sum_{n=1}^N\zeta_n\prod_{k=n+1}^{N+1}(1+\vert\zeta_k\vert^2)s_n(
\zeta_n,..,\zeta_{N+1})$$ with
$$s_n(\zeta_n,..,\zeta_{N+1})=\sum c_{\vec {i},\vec {j}}^{(N+1)}\zeta_{
i_1}\bar{\zeta}_{j_1}\zeta_{i_2}\bar{\zeta}_{j_2}..\zeta_{i_r}\bar{
\zeta}_{j_L},$$ the sum is over multiindices satisfying
$$\begin{matrix} &&j_1&\le&..&\le j_L&\le&N+1\\
&&\lor&&&\lor\\
n&\le&i_1&\le&..&i_L&\end{matrix}
,\quad\sum_{l=1}^L(j_l-i_l)=n-1,$$ and $c_{\vec {i},\vec
{j}}^{(N+1)}$  can be computed, in principle, recursively. If
$j_L\le N$, then $c_{i,j}^{(N+1)}=c_{i,j}^{(N)}$. Otherwise the
index $(i,j)$ has the form
$$\begin{array}{ccccccccccc}&&j_1&\le&..&j_r&<&N+1&..&.&N+1\\
&&\lor&&&\lor&&\lor&&&\lor\\i_0&\le&i_1&\le&..&i_r&\le&i_{r+1}&..&\le&i_L\end{array}$$
where $r+s=L$. The corresponding terms will all originate from the
term involving the index $s$ in the last expression for
(\ref{sum1}). There are many ways that terms could arise, and at
best we obtain a formula for $c^{(N+1)}$ in terms of coefficients
$c^{(N)}$. So at this point we can only see that these
coefficients are positive.
\end{proof}

Our aim now is to consider another approach which yields a closed
formula for "generic" $c_{i,j}$. This formula a priori involves
signs, and we will make use of Proposition \ref{positivity} to
identify cancellations.

The matrix
$$\left(\begin{matrix} 1&\sum_{j=1}^nx^{*}_jz^{-j}\\
0&1\end{matrix} \right)\left(\begin{matrix} a_2&0\\
0&a_2^{-1}\end{matrix} \right)\left(\begin{matrix}\alpha (z)&\beta (z)\\
\gamma (z)&\delta (z)\end{matrix} \right)$$
$$=\left(\begin{matrix} a_2\alpha+x^{*}a_2^{-1}\gamma&a_2\beta +x^{
*}a_2^{-1}\delta\\
a_2^{-1}\gamma&a_2^{-1}\delta\end{matrix} \right)$$ is special
unitary, for all $z\in S^1$. Therefore $-\gamma^{*}=a_2^2\beta
+x^{*}\delta,$ and initially assuming $\delta$ is nonvanishing,
this implies $x^{*}=P_{-}(-\gamma^{*}\delta^{-1}).$ In particular
$$x^{*}_1=Residue(-\gamma^{*}\delta^{-1})$$
$$=-\gamma_1^{*}+(\gamma^{*}_2\delta_1+\gamma_3^{*}\delta_2+..)-(
\gamma_3^{*}(\delta^2)_2+..)$$
$$=-\sum_{m\ge 1}\gamma_m^{*}\sum (-1)^s\delta_{n_1}..\delta_{n_s}$$
where the second sum is over tuples $n_1,..,n_s\ge 1$ satisfying
$\sum n_l=m-1$. Using the formulas for $\gamma^{*}$ and $ \delta$
in Proposition \ref{gammadelta},
$$x_1^{*}=\sum (-1)^{s+1}((-1)^{r_m+1}\sum\zeta_{i_{m,1}}\bar{\zeta}_{
j_{m,1}}...\zeta_{i_{m,r_m}}\bar{\zeta}_{j_{m,r_m}}\zeta_{i_{m,r_
m+1}})$$
$$(-1)^{r_{n_1}}(\sum\zeta_{i_{n_1,1}}\bar{\zeta}_{j_{n_1,1}}...\zeta_{
i_{n_1,r_{n_1}}}\bar{\zeta}_{j_{n_1,r_{n_1}}})$$
$$...(-1)^{r_{n_s}}(\sum\zeta_{i_{n_s,1}}\bar{\zeta}_{j_{n_s,1}}.
..\zeta_{i_{n_s,r_{n_s}}}\bar{\zeta}_{j_{n_s,r_{n_s}}})$$ where
the indexing can be described in the following way: the first sum
is over $m,n_1,.,n_s\ge 1$ satisfying $\sum_ln_l=m-1$, the first
internal sum, or cluster indexed by $m$, is over indices
satisfying
$$0<i_{m,1}<j_{m,1}<..<j_{m,r}<i_{m,r_m+1},\quad\sum_{k=1}^{r_m+1}
i_{m,k}-\sum_{k=1}^{r_m}j_{m,k}=m$$ and the cluster indexed by
$n_l$ is over indices satisfying
$$0<i_{n_l,1}<j_{n_l,1}<..<j_{n_l,r_{n_l}},\quad\sum_{k=1}^{r_{n_
l}}(j_{n_l,k}-i_{n_l,k})=n_l.$$

We now write this as a single sum and consider one of the terms.
We can put the $i$-indices (which are organized in clusters)
$$i_{m,1},..,i_{m,r_m+1};i_{n_1,1},..,i_{n_1,r_{n_1}};..;i_{n_s,1}
,..,i_{n_s,r_{n_s}}$$ and the $j$-indices
$$j_{m,1},..,j_{m,r_m};j_{n_1,1},..,j_{n_1,r_{n_1}};..;j_{n_s,1},
..,j_{n_s,r_{n_s}}$$ in nondecreasing order, which we write as
$$\mathbf i_0\le \mathbf i_1\le ..\le \mathbf i_L\quad and\quad \mathbf j_
1\le ...\le \mathbf j_L,$$ respectively.

\begin{lemma}\label{strict} In addition to being nondecreasing,
the indices $\mathbf i_l$, $\mathbf j_l$ satisfy $\mathbf
i_{l-1}<\mathbf j_l,$ for $l=1,..,L$.
\end{lemma}

\begin{proof} With the possible
exception of $i_{m,r+1}$, for any given $\mathbf i$-index, it is
possible to find a $\mathbf j$-index with greater value, so that
the map from these $\mathbf i$-indices to $\mathbf j$-indices is $
1-1$ (simply map $i_{n,l}$ to $j_{n,l}$). One of $\mathbf i_{L-1}$
or $\mathbf i_ L$ must be strictly less than $\mathbf j_L$, hence
$\mathbf i_{L-1}$ must be strictly less than $\mathbf j_L$.
Similarly one of $\mathbf i_{L-2}$ or $\mathbf i_{L -1}$ or
$\mathbf i_L$ must be strictly less than $\mathbf j_{L-1}$, hence
$\mathbf i_{L-2}$ must be strictly less than $\mathbf j_{L-1}$.
Continuing in this way, this implies the strict inequalities in
the Lemma.
\end{proof}

We claim that we can additionally assume that
\begin{equation}\label{inequal}\mathbf i_l \le \mathbf j_l,\quad
l=1,..,L.\end{equation} This is not implied by cluster
decomposition considerations. For example the index set
$$\begin{matrix} &2&2\\
1&1&3\end{matrix} $$ violates (\ref{inequal}), yet there are two
cluster decompositions: $1<2<3;1<2$ (with
$(-1)^{s+L}=(-1)^{1+2}=-1$) and $3;1<2;1<2$ (with
$(-1)^{s+L}=(-1)^{2+2}=1$). This claim is justified by Proposition
\ref{positivity}, which implies that terms corresponding to
indices not satisfying (\ref{inequal}) will cancel out (It would
clearly be desirable to see this cancellation directly, but I do
not know how to do this). This implies the following formula.

\begin{lemma}\label{formula1}$x_1^{*}=\sum \mathbf c_{\mathbf i,\mathbf j}\zeta_{\mathbf i_0}\zeta_{\mathbf i_
1}\bar{\zeta}_{\mathbf j_1}..\zeta_{\mathbf
i_L}\bar{\zeta}_{\mathbf j_ L},$ where the indices satisfy the
constraints \begin{equation}\label{index2}0<\mathbf i_0\le \mathbf
i_1\le ..\le \mathbf i_L,\quad \mathbf j_1\le ...\le \mathbf
j_L,\quad \mathbf i_1\le \mathbf j_1,..,\mathbf i_L\le \mathbf j_
L,\end{equation}
$$\mathbf i_0<\mathbf j_1,..,\mathbf
i_{L-1}<\mathbf j_L,\quad\sum \mathbf i-\sum \mathbf j=1,$$ and
\begin{equation}\label{csum}\mathbf c_{\mathbf i,\mathbf j}=\sum (-1)^{s+L},\end{equation}
where the sum is over all possible ways in which the indices can
be partitioned as
$$i_{m,1},..,i_{m,r_{m+1}};i_{n_1,1},..,i_{n_1,r_{n_1}};..;i_{n_s
,1},..,i_{n_s,r_{n_s}}$$
$$j_{m,1},..,j_{m,r_m};j_{n_1,1},..,j_{n_1,r_{n_1}};..;j_{n_s,1},
..,j_{n_s,r_{n_s}}$$ so that the strict interlacing inequalities
$$0<i_{m,1}<j_{m,1}<..<j_{m,r}<i_{m,r+1},\quad\sum_ki_{m,k}-\sum_
kj_{m,k}=m$$
 and
$$0<i_{n_l,1}<j_{n_l,1}<..<j_{n_l,r},\quad\sum_k(j_{n_l,k}-i_{n_l
,k})=n_l$$ hold for $l=1,..,s$.
\end{lemma}

To compare with the formula in Proposition \ref{positivity}, we
first sum over $n=\mathbf i_0$, and write
\begin{equation}\label{formula2}x_1^{*}=\sum_{n=1}^{\infty}\zeta_n\sum \mathbf
c_{(n,\mathbf i),\mathbf j}\zeta_{ \mathbf
i_1}\bar{\zeta}_{\mathbf j_1}..\zeta_{\mathbf i_L}\bar{\zeta}_{
\mathbf j_L}\end{equation} where $(n,\mathbf i)$ now stands for
$n\le \mathbf i_1\le ..\le \mathbf i_L$. This implies

\begin{equation}\label{formula3}\sum \mathbf
c_{(n,\mathbf i),\mathbf j}\zeta_{ \mathbf
i_1}\bar{\zeta}_{\mathbf j_1}..\zeta_{\mathbf i_L}\bar{\zeta}_{
\mathbf j_L} =\left(\prod_{k=n+1}^{\infty}
(1+\vert\zeta_k\vert^2)\right) \sum
c_{i,j}\zeta_{i_1}\bar{\zeta}_{j_1}\zeta_{
i_2}\bar{\zeta}_{j_2}..\zeta_{i_r}\bar{\zeta}_{j_r}
\end{equation}
where the indexing set for the latter sum satisfies the
constraints in Proposition \ref{positivity}. To directly compare
the coefficients we expand the product of factors $(1+\vert
\zeta_j \vert^2)$ and distribute the pairs $\zeta_j$ and
$\bar{\zeta_j}$. This implies the following

\begin{lemma}\label{formula5} Consider an index as in (\ref{index2}), with $n=\mathbf
i_0$.

(a) If $\{\mathbf i_l\}\cap\{\mathbf j_{l^{\prime}}\}$ is null,
then $\mathbf c_{(n,\mathbf i),\mathbf j}=c_{\mathbf i,\mathbf
j}$.

(b) In general
$$\mathbf c_{(n,\mathbf i),\mathbf j}=\sum c_{i,j},$$
where the sum is over all subindexing sets of $(n,\mathbf
i,\mathbf j)$, resulting from cancellation of pairs $\mathbf
i_l=\mathbf j_{l^{\prime}}$, which satisfy the constraints in
Proposition \ref{positivity}.

(c) In particular for any indexing set $(i,j)$ as in Proposition
\ref{positivity}, $c_{i,j}\le \mathbf c_{(n,i),j}$.
\end{lemma}

\begin{example} To clarify (b), given an indexing set such as
$$\begin{matrix} &5&6&7\\
3&4&5&6\end{matrix} $$ there are three proper subindexing sets,
$$\begin{matrix} &6&7\\
3&4&6\end{matrix} \quad\quad \begin{matrix} &5&7\\
3&4&5\end{matrix} \quad\quad \begin{matrix} &7\\
3&4\end{matrix} $$
\end{example}

Part (a) of Lemma \ref{formula5}, and Lemma \ref{formula1}, yield
an expression for a generic $c_{i,j}$, where generic is defined by
the null intersection condition in (a). Using this formula it is
possible to write ``most'' of the terms in $s_{n,r}$ in
Proposition \ref{positivity} in terms of products  of the
Hermitian expressions
$$b_n(m)=\zeta_n \bar{\zeta}_{n+m}+\zeta_{n+1}
\bar{\zeta}_{n+1+m}+..$$ These expressions can be estimated using
Cauchy-Schwarz, and they are also easy to understand in
probabilistic contexts. Unfortunately I do not know how to
systematically estimate nongeneric terms.

\begin{example}
$$s_2=s_{2,1}=b_2(1)+b_3(1)$$
and in general
$$s_{n,1}=b_n(n-1)+b_{n+1}(n-1)$$
$s_{3,2}$ is a quadratic expression in terms of the variables
$\zeta_3\bar{\zeta_4}$, $\zeta_4\bar{\zeta_5}$,.. The matrix is
$$\begin{matrix}1&3&2&2&2&..\\
&3&6&4&4&4&..\\&&3&6&4&4&4&..\\&&&3&6&4&4&4&..\end{matrix}$$
Therefore $$s_{3,2}=b_3(1)^2+b_4(1)^2+\sum_{i\ge 4}\zeta_i
\bar{\zeta}_{i+1}\zeta_i \bar{\zeta}_{i+1}+\zeta_3
\bar{\zeta}_{4}\zeta_4 \bar{\zeta}_{5}+2\sum_{i\ge 4}\zeta_i
\bar{\zeta}_{i+1}\zeta_{i+1} \bar{\zeta}_{i+2}$$ Thus ``most" of
$s_{3,2}$ can be written in terms of powers of Hermitian
expressions, and two ``diagonal" sums near the boundary of the
cone that we are adding over.
\end{example}

\subsection{$\zeta$ in terms of $x$}

We have $\zeta_n=\zeta_1(x_n,x_{n+1},..)$, and for a finite number
of variables, one can generate formulas for $\zeta_1$. For
example, if $p_n=\prod_{j>n}(1+\vert \zeta_j \vert^2)$, then
$$\zeta_1(x_1,x_2,x_3,x_4)=
=\frac 1{p_1}x_1-\frac 1{p_1p_2p_3}x_2^2\bar {x}_3+2\frac 1{p_1
p_2p_3^2p_4}x_2x_3^2\bar {x}_3\bar {x}_4-2\frac 1{p_1p_3p_4}x_2x_
3\bar {x}_4$$
$$-\frac 1{p_1p_2p_3^3p_4^2}x_3^4\bar {x}_3\bar {x}_4^2+\frac 1{p_
1p_3^2p_4^2}x_3^3\bar {x}_4^2,$$ where the $p_i$ can be expressed
in terms of $x$ using the displayed line following (6.10) in
\cite{P}. But I have not made any progress toward finding a
general formula.

\end{document}